\documentclass[11pt,reqno]{amsart}
\usepackage{amssymb,latexsym}
\usepackage{amsmath}
\usepackage{amsthm}
\usepackage{graphicx}
\usepackage{hyperref}
\usepackage{titletoc}
\usepackage{pdfsync}
\usepackage{color,xcolor}
\usepackage{amsthm,amsmath,amssymb}

\usepackage{mathrsfs}
\def\red{\color{red}}

\numberwithin{equation}{section}
\newtheorem{theorem}{Theorem}[section]
\newtheorem{proposition}[theorem]{Proposition}
\newtheorem{lemma}[theorem]{Lemma}

\usepackage{multirow}
\usepackage[font=bf,aboveskip=15pt]{caption}
\usepackage[toc,page]{appendix}

\theoremstyle{definition}

\newtheorem{remark}[theorem]{Remark}

\newtheorem{lem}[theorem]{Lemma}

\newcommand{\inn}{{\quad\hbox{in } }}

\def\R{{\mathfrak R}}

\def\begeq{\begin{equation}}
\def\endeq{\end{equation}}

\def\R{\Bbb R}

\allowdisplaybreaks

\begin{document}

\title{Doubling the equatorial for the prescribed scalar curvature
problem on ${\mathbb{S}}^N$ }

\author{Lipeng Duan}
\address{Lipeng Duan,
\newline\indent School of Mathematics and Information Science, Guangzhou University,
\newline\indent Guangzhou 510006,  P. R. China.
}
\email{lpduan777@sina.com}

\author{Monica Musso}
\address{Monica Musso,
\newline\indent Department of Mathematical Sciences, University of Bath,
\newline\indent Bath BA2 7AY, United Kingdom
}
\email{mm2683@bath.ac.uk}

\author{Suting Wei}
\address{Suting Wei,
\newline\indent Department of Mathematics, South China
Agricultural University,
\newline\indent Guangzhou, 510642, P. R. China.
}
\email{stwei@scau.edu.cn}

\begin{abstract}
We consider the prescribed scalar curvature
problem on  $ {\mathbb{S}}^N  $
\begin{align*}
\Delta_{{\mathbb S}^N} v-\frac{N(N-2)}{2} v+\tilde{K}(y) v^{\frac{N+2}{N-2}}=0 \quad \mbox{on} \ {\mathbb S}^N, \qquad v >0 \quad \inn{\mathbb S}^N,
\end{align*}
under the assumptions that the scalar curvature $\tilde K$
is rotationally symmetric, and has a positive local maximum point between the
poles. We prove the existence of infinitely many  non-radial positive solutions, whose energy can be made arbitrarily
large. These solutions are invariant under some non-trivial sub-group of $O(3)$ obtained doubling the equatorial. We use
the finite dimensional  Lyapunov-Schmidt reduction method.

\vspace{2mm}


\vspace{2mm}

\end{abstract}

\date{\today}


\maketitle

\section{Introduction} \label{sec1}
 Given the  $N $-th sphere  $({\mathbb S}^N, g) $ equipped with the standard metric $g$ and a fixed smooth function  $\tilde{K} $, the prescribed scalar curvature problem on ${\mathbb S}^N$ consists in  understanding whether  it is possible to find another metric  $ \tilde{g} $ in the conformal class of  $g $, such that the scalar curvature of  $ \tilde{g} $ is  $\tilde{K} $.
 For a conformal change of the metric
 $$
 \tilde g = v^{4 \over N-2} \, g,
 $$
 for some positive function $v:  {\mathbb S}^N \to \R$, the scalar curvature with respect to $\tilde g$ is given by
 $$
 v^{-{N+2 \over N-2} } \, \left( \Delta_{{\mathbb S}^N} v-\frac{N(N-2)}{2} v \right)
 $$
 where $\Delta_{{\mathbb S}^N}$ is the Laplace-Beltrami operator on ${\mathbb S}^N$.
 Thus the prescribed scalar curvature problem on ${\mathbb{S}}^N$
 can be addressed by studying the solvability of the problem
 \begin{equation}
\label{1.4n}
\Delta_{{\mathbb S}^N} v-\frac{N(N-2)}{2} v+\tilde{K}(y) v^{\frac{N+2}{N-2}}=0 \quad \mbox{on} \ {\mathbb S}^N, \qquad v >0 \quad \inn {\mathbb S}^N.
\end{equation}

\medskip
Testing the equation against $v$ and integrating on ${\mathbb{S}}^N$, we get that
a necessary condition for the solvability of this problem is that  $\tilde{K}(y) $ must be positive somewhere. There are other obstructions for the existence of  solutions in the energy space, which are said to be of topological type. For instance, a solution  $v $ must satisfy the following Kazdan-Warner type condition (see \cite{KW}):
 \begin{equation}
\label{ka}
\int_{{\mathbb S}^N} \nabla \tilde{K}(y) \cdot \nabla y \, v^{\frac{2N}{N-2}} \, d\sigma =0.
\end{equation}
This condition is a direct consequence of Theorem 5.17 in \cite{KW1}: Kazdan and Warner proved that given a positive solution $v$ to
$$
	\Delta_{{\mathbb S}^N} v-\frac{N(N-2)}{2} v+H(y) v^{a}=0
$$
on the standard sphere ${\mathbb{S}}^N$, $N\geq 3$, then
$$
\int_{{\mathbb{S}}^N} v^{a+1} \nabla H \cdot \nabla F = {1\over 2} (N-2) \left( {N+2 \over N-2}- a \right) \int_{{\mathbb{S}}^N} v^{a+1} H \, F \,
$$
for any spherical harmonics $F$ of degree $1$. Taking $a={N+2 \over N-2}$, $H= \tilde K$ and $F=y$,  we obtain \eqref{ka}.

The problem of determining which  $\tilde{K}(y) $ admits a solution has been the object of several studies in the past years. We refer the readers to \cite{BC91JFA,B-96-CPDE,BP,CY-91-DM,
CL-98-JDG, ChenDing, D1,H-90-DM, KW,SZ-96-CVPDE}, and the references therein.



\medskip
The prescribed scalar curvature problem on ${\mathbb{S}}^N$ \eqref{1.4n} can be transformed into a semi-linear elliptic equation in the flat space $\R^N$ via the stereo-graphic projection. Letting $D^{1,2}(\mathbb{R}^N) $ denote the completion of  $C_c^\infty(\R^N)  $ with respect to the norm  $\int_{\mathbb{R}^N}|\nabla v|^2 $, solving \eqref{1.4n} in the energy space is equivalent to solving
 \begin{align}\label{Prob1}
\Delta v+ K(y) v^{2^*-1} =0, \quad v >0,  \quad \text{in}~ \mathbb{R}^N,\,
v\in D^{1,2}(\mathbb{R}^N).
\end{align}

\medskip
Because of its geometry roots, it is of interest to establish under what kind of assumptions on $K$ problem  \eqref{Prob1} admits one or multiple solutions.

For  $N=3 $, Y.Y. Li \cite{L-93-CPAM} showed problem \eqref{Prob1} has infinitely many solutions provided that  $ K(y) $ is bounded below, and periodic in one of its variables, and the set  $\{x\, | \, K(x)=\max_{y\in \R^3}K(y)\} $ is not empty and contains at least one bounded connected component.

If $K$ has the form $K(y)=1+\epsilon h(y) $, namely it is   a perturbation of  the constant $1 $, D. Cao, E. Noussair and S. Yan \cite{CEY-02-CPDE} proved the existence of multiple solutions.

If  $K(y) $ has a sequence of strictly local maximum points moving to infinity, S. Yan \cite{Y-00-JDE} constructed infinitely many solutions.

In \cite{WY-10-JFA}, J. Wei and S. Yan showed that problem \eqref{Prob1} has infinitely many solutions provided  $K $ is radially symmetric $K(y) = K(r)$, $r= |y|$, and has a local maximum around a given $r_0 >0$. More precisely, they ask that there are  $ r_0  $,  $c_0> 0 $ and $ m \in [2, N-2)  $ such that
\begin{equation*}
K(s)= K(r_0)-c_0|s-r_0|^m+O\big({|s-r_0|^{m+\sigma}}\big),
 \quad s\in(r_0-\delta, r_0+\delta),
\end{equation*}
for some  $\sigma, \delta $  small positive constants.  These solutions are obtained by gluing together a large number of  Aubin-Talenti bubbles (see \cite{Talenti})
\begin{align*}
U_{x, \Lambda}(y)= c_N \Big(\frac{\Lambda}{1+\Lambda^2|y-x|^2}\Big)^{\frac{N-2}{2}}, \, c_N= [N(N-2)]^{\frac{N-2}{4}}.
\end{align*}
For any $x \in \R^N$ and $\Lambda \in \R^+$, these functions solve
\begin{equation}\label{criticalequation}
	\Delta u+  u^{N+2 \over N-2}=0 \quad ~ \text{in}~ \mathbb{R}^N.
\end{equation}
In fact they are the only positive solutions to \eqref{criticalequation}. At main order their solution looks like
\begin{equation*}
	\tilde  u_k \sim \sum_{j=1}^k U_{x_j, \bar \Lambda},
\end{equation*}
where  $ \bar \Lambda $ is a positive constant and the points $x_j$ are distributed along the vertices of a regular polygon of $k$ edges in the $(y_1,y_2)$-plane, with $|x_j| \to r_0$ as $k \to \infty$:
\begin{equation*}
	x_j=\Big(\tilde r \cos\frac{2(j-1)\pi}k, \tilde r\sin\frac{2(j-1)\pi}k, 0,\cdots,0\Big),\quad j=1,\cdots,k,
\end{equation*}
with  $ \tilde r \rightarrow  r_0  $ as  $ k \rightarrow \infty $.

Under the weaker symmetry conditions for  $K(y)=K(|y'|,y'') $ with  $y=(|y'|,y'')\in \R^2\times \R^{N-2} $, S. Peng, C. Wang and S. Wei \cite{PWW-19-JDE} constructed infinitely many bubbling solutions, which concentrate at the saddle points of the potential  $K(y) $. Their method uses local Pohozaev identities adapted to  problem \eqref{Prob1}. Y. Guo and B. Li \cite{Gyslb} generalized these results for problems \eqref{Prob1} with polyharmonic operators \cite{WY-10-JFA}. For the fractional scalar field equation, we refer to \cite{Gyxnjj,LPY-16-DCDS}.

The study of other aspects of problem \eqref{Prob1}, such as radial symmetry of their solutions, uniqueness of solutions, Liouville type theorem, a prior estimates, and bubbling analysis, have been the object of investigation of several researchers. We refer the readers to the papers \cite{AAP,DLY, L-93-JFA, L-96-CPAM,  LWX,  LL-90-AA, N-82-In,NY-01-NA,Y-00-JDE} and the references therein.

\medskip
 Recently, Y. Guo, M. Musso, S. Peng and S. Yan \cite{GMPY} investigated the spectral property of the linearized problem  associated to \eqref{Prob1} around  the solution $ \tilde u _k  $ found in \cite{WY-10-JFA}. They  proved a non-degeneracy result for such operator  using a refined version of local Pohozaev identities. As an application of this  non-degeneracy result, they built new type of solutions by gluing another large number of bubbles, whose centers lie near the circle  $|y|= r_0 $ in the  $(y_3,y_4) $-plane.

\medskip
All these results concern solutions made by gluing together Aubin-Talenti bubbles with centres distributed along the vertices of one or more planar polygons, thus of two-dimensional nature.
The purpose of this paper is to present a different type of solution to \eqref{Prob1} with a more complex concentration structure, which cannot be reduced to a two-dimensional one.


\medskip
To present our result,  we made the following assumptions on $K$: it is radially symmetric and \\[2mm]
  ${\bf H}  $ : There are  $ r_0  $ and  $c_0> 0 $ such that
\begin{equation*}
K(s)= K(r_0)-c_0|s-r_0|^m+O\big({|s-r_0|^{m+\sigma}}\big),
 \quad s\in(r_0-\delta, r_0+\delta),
\end{equation*}
 where  $\sigma, \delta >0 $ are small constants,
\begin{align}\label{assumptionform}
 m \in \begin{cases}
[ 2, N-2) \quad  &\text{if} \quad N=5 ~\text{or} ~6,
  \\[2mm]
  \big(\frac{N-2}{2},  N-2\big) \quad  &\text{if} \quad N \ge 7.
\end{cases}
\end{align}
There is a slight difference between  our assumptions on  $ K(s) $  and the ones  in \cite{WY-10-JFA}. We will comment on this issue later in the introduction.

\medskip
Without loss of generality, we assume  $ r_0= 1, \,   K(1) =1$. For any integer $k$, we define $${\bf r}=k^{\frac{N-2}{N-2-m}} .$$  We  scale $ u(y)={\bf r}^{-\frac{N-2}{2}}v \big(\frac{|y|}{{\bf r}}\big) $, so that problem \eqref{Prob1} becomes
\begin{align}\label{Prob2}
-\Delta u=  K\Big(\frac{|y|}{{\bf r}}\Big) u^{2^*-1}, \quad u >0,  \quad \text{in}~ \mathbb{R}^N,
\quad
u \in D^{1,2}(\mathbb{R}^N),
\end{align}
where $$2^*= {2N \over N-2}.$$
Consider the points
\begin{align*}
 \begin{cases} \overline{x}_j= r \Big(\sqrt {1-h^2} \cos{\frac{2(j-1) \pi}{k}},  \sqrt {1-h^2} \sin{\frac{2(j-1) \pi}{k}}, h,{\bf{0}}\Big), \quad j= 1, \cdots, k,
 \\[4mm]
 \underline{x}_j= r \Big(\sqrt {1-h^2}  \cos{\frac{2(j-1) \pi}{k}},  \sqrt {1-h^2} \sin{\frac{2(j-1) \pi}{k}},-h, {\bf{0}}\Big), \quad j= 1, \cdots, k,
 \end{cases}
\end{align*}
where  $ {\bf{0}} $ is the zero vector in  $\mathbb{R}^{N-3} $ and  $ h, r  $ are positive parameters.

We define
 \begin{align}\label{Wrh1}
W_{r,h,\Lambda} (y)
&= \sum_{j=1}^kU_{\overline{x}_j, \Lambda}  (y) +\sum_{j=1}^k U_{\underline{x}_j, \Lambda}  (y), \quad y \in \R^N.
\end{align}
We will produce a family of solutions to \eqref{Prob2} which are small perturbations of $W_{r,h,\Lambda}$, for any integer $k$ sufficiently large. The Aubin-Talenti bubbles are now centred at points lying on the top and the bottom circles of a cylinder and this configuration is now invariant under a non-trivial sub-group of $O(3)$ rather than $O(2)$.

Throughout of the present paper, we assume  $N\ge 5$ and $(r,h, \Lambda) \in{{\mathscr  S}_k} $, where
\begin{align}\label{definitionofsk}
 {{\mathscr S}_k}
= \Bigg\{&(r,h,\Lambda) \big|\,  r\in \Big[k^{\frac{N-2}{N-2-m}}-\hat \sigma, k^{\frac{N-2}{N-2-m}}+\hat  \sigma \Big], \quad
\Lambda \in \Big[\Lambda_0-\hat   \sigma,  \Lambda_0+\hat \sigma\Big], \nonumber
\\[2mm]
& \qquad \qquad
 h \in \Big[\frac{B'}{k^{\frac{N-3}{N-1}}}  \Big(1-\hat   \sigma\Big),
\frac{B'}{k^{\frac{N-3}{N-1}}} \Big(1+\hat \sigma\Big)\Big] \Bigg\},
 \end{align}
 with  $ \Lambda_0,  B' $ are the constants in  \eqref{lambda0}, \eqref{B'} and  $ \hat \sigma$ is a small fixed number, independent of $k$. For  convenience, we denote
   $$ {\bf r}=k^{\frac{N-2}{N-2-m}},\quad{\bf h}= \frac{B'}{k^{\frac{N-3}{N-1}}}. $$
   Since $h \to 0$ as $k \to \infty$, the two circles where the points $\overline x_j$ and $\underline x_j$ are distributed become closer to each other as $k$ increases.

%

\medskip
In this paper, we prove that for any  $k $ large enough problem \eqref{Prob2} has  a family of solutions $u_k$ with the approximate form
\begin{equation}\label{approximatesolution2}
u_k(y) \sim W_{r,h,\Lambda}.
\end{equation}
These solutions   have polygonal symmetry in the  $(y_1,y_2) $-plane, are even in the  $y_3 $ direction and radially symmetric in the variables  $y_4,\cdots,y_N $.  Our solutions are thus different from the ones obtained in \cite{WY-10-JFA} and have strong analogies with the doubling construction of the entire finite energy sign-changing solutions  for the Yamabe equation in \cite{medina}.

\medskip
 Define the symmetric Sobolev space:
\begin{align*}
H_s= \Bigg\{& u : u \in H^1(\mathbb{R}^N),  \text{ $u $  is even in  $ y_\ell, \, \ell= 2, 3, 4, \cdots, N, $} \quad  \nonumber
\\[2mm]
&\qquad u \Big(\sqrt{y_1^2+y_2^2} \cos \theta, \sqrt{y_1^2+y_2^2} \sin \theta, y_3, y''\Big)
\\[2mm]
& \qquad= u \Big(\sqrt{y_1^2+y_2^2} \cos {\big(\theta+\frac{2j\pi}{k}\Big)}, \sqrt{y_1^2+y_2^2} \sin {\big(\theta+\frac{2j\pi}{k}\Big)}, y_3, y''\Big)  \Bigg\},
\end{align*}
where  $ \theta= \arctan{\frac{y_2}{y_1}} $.
\medskip
Let us define the following norms which capture the decay property of functions
\begin{equation}\label{star}
\|u\|_{*}= \sup_{y \in \R^N}\Bigl(
\sum_{j=1}^k \Big[\frac1{(1+|y-\overline{x}_j|)^{\frac{N-2}{2}+\tau
}}+\frac1{(1+|y-\underline{x}_j|)^{\frac{N-2}{2}+\tau
}}\Big] \Bigr)^{-1}|u(y)|,
\end{equation}
and
\begin{equation}\label{starstar}
\|f\|_{**}= \sup_{y \in \R^N}\Bigl(\sum_{j=1}^k \Big[\frac1{(1+|y-\overline{x}_j|)^{\frac{N+2}2+\tau
}}+\frac1{(1+|y-\underline{x}_j|)^{\frac{N+2}2+\tau
}}\Big]\Bigr)^{-1}|f(y)|,
\end{equation}
where  \begin{align}  \label{tau and epsilon}
\tau= (\frac{N-2-m}{N-2},  \frac{N-2-m}{N-2}+ \epsilon_1),
\end{align}
 for some $\epsilon_1$ small.
The main results of this paper are  the following:

\medskip
\begin{theorem}\label{main1}
Suppose that $K(|y|)$ satisfies ${\bf H}$ and  $N\ge 5$. Then there exists a large integer $k_0$, such that for each integer $k\geq k_0$, problem \eqref{Prob2} has a solution  $u_k$ of the form
\begin{align}\label{u_k}
u_k(y)= W_{r_k, h_k,\Lambda_k}(y)+\phi_k(y),
\end{align}
where  $ \phi_k \in  H_s, \, (r_k,s_k, \Lambda_k) \in {{\mathscr  S}_k}$ and  $ \phi_k  $ satisfies
\begin{align*}
\|\phi _k\|_{*} =  o_k(1), \quad {\mbox {as}} \quad k \to \infty.
\end{align*}
Equivalently,  problem \eqref{Prob1} has solution  $v_k(y)  $ of the form
\begin{align*}
v_k(y)={\bf r}^{\frac{2-N} 2} \Big[W_{r_k, h_k,\Lambda_k}({\bf r} y)+\phi_k({\bf r} y)\Big].
\end{align*}

\end{theorem}

%

 \medskip
Let us outline the main ideas in the proof of Theorem \ref{main1}. The first step in our argument is to find $\phi$ so that
$u=W_{r,h,\Lambda}+\phi$ solves the auxiliary  problem
\begin{equation}\label{re}
	\begin{cases}
		-\Delta \bigl( W_{r,h,\Lambda}+\phi \bigr)
		=K\bigl(\frac{|y|}{{\bf r}}\bigr)\bigl(W_{r,h,\Lambda}+\phi \bigr)^{2^*-1}
		\\[2mm]
		\qquad \qquad \qquad \qquad+\sum\limits_{\ell=1}^3 \sum\limits_{j=1}^k {c_\ell}
		\Big(\, U_{\overline{x}_j,\Lambda}^{2^*-2}\overline{\mathbb{Z}}_{\ell j}+U_{\underline{x}_j,\Lambda}^{2^*-2}\underline{\mathbb{Z}}_{\ell j}\,\Big)\, \; \text{in}\;
		\R^N,\\[2mm]
		\phi \in  \mathbb{E},
	\end{cases}
\end{equation}
for some scalars $ c_\ell  $ for  $ \ell= 1,2,3 $. In \eqref{re}, the functions $\overline{\mathbb{Z}}_{\ell j}$ and $\underline{\mathbb{Z}}_{\ell j}$ are given by
\begin{align*}
&\overline{\mathbb{Z}}_{1j}=\frac{\partial U_{\overline{x}_j, \Lambda}}{\partial r},
\qquad \qquad
\overline{\mathbb{Z}}_{2j}=\frac{\partial U_{\overline{x}_j, \Lambda}}{\partial h},
\qquad \qquad
\overline{\mathbb{Z}}_{3j}=\frac{\partial U_{\overline{x}_j, \Lambda}}{\partial \Lambda}, \nonumber
\\[2mm]
& \underline{\mathbb{Z}}_{1j}=\frac{\partial U_{\underline{x}_j, \Lambda}}{\partial r},
\qquad \qquad
\underline{\mathbb{Z}}_{2j}=\frac{\partial U_{\underline{x}_j, \Lambda}}{\partial h},
\qquad \qquad
\underline{\mathbb{Z}}_{3j}=\frac{\partial U_{\underline{x}_j, \Lambda}}{\partial \Lambda},
\end{align*}
for $ j= 1, \cdots, k $. Moreover the function $\phi$ belongs to the set $\mathbb{E}$ given by
\begin{align}\label{SpaceE}
\mathbb{E}= \Big\{v:  v\in H_s, \quad &\int_{{\mathbb{R}}^N} U_{\overline{x}_j, \Lambda}^{2^*-2}  \overline{\mathbb{Z}}_{\ell j} v= 0 \quad\text{and}  \nonumber
\\[2mm]
\quad & \int_{{\mathbb{R}}^N} U_{\underline{x}_j, \Lambda}^{2^*-2}  \underline{\mathbb{Z}}_{\ell j} v= 0, \quad j= 1, \cdots, k, \quad \ell= 1,2, 3  \Big\}. \end{align}
From the linear theory developed in Section \ref{sec3}, \eqref{re} can be solved by means of the contraction mapping theorem. More precisely, we prove that, for any $(r,h,\Lambda ) \in  {{\mathscr S}_k}$ there exist $\phi = \phi_{r,h,\Lambda} \in \mathbb{E} $ and constants $c_\ell$, $\ell =1,2,3$  which solve the auxiliary problem \eqref{re}.

The  results in Theorem \ref{main1} can be then proved if there exists a choice of $(r,h,\Lambda ) \in {{\mathscr S}_k} $ so that the multipliers  $c_\ell\,(\ell= 1,2,3) $ in \eqref{re} can be made equal zero.
We now make the following observation:  take the functional corresponding to the problem \eqref{Prob2}
\begin{equation}\label{energyfunctial}
	I(u)=\frac{1}{2}\int_{{\mathbb{R}}^N}|\nabla u|^2 dy-\frac{1}{2^{*}}\int_{\R^N}K\Big(\frac{|y|}{{\bf r}}\Big)|u|^{2^*} dy.
\end{equation}
 If we can choose   $(r ,h,\Lambda)  $ to be a critical point of function
\[ F(r,h,\Lambda) := I(W_{r, h, \Lambda}+\phi_{r, h, \Lambda})
\quad \text{for}  ~ \phi \in{\mathbb{E}}, \]
then the constants  $ c_ \ell,\,\ell=1,2,3 $ would be zero. Thus finding solutions of problem \eqref{Prob2} would be reduced to find a critical point of  $ F(r,h,\Lambda) $. This is the result in Proposition \ref{pro2.4}.

An important work of this paper is to give an accurate expression of  $ F(r,h,\Lambda) $ (see Proposition \ref{pr0position2,6}).
Under the assumptions  $ r \sim k^{\frac{N-2}{N-2-m}},  h \rightarrow 0, \frac 1 {hk} \rightarrow 0 $ as  $k \rightarrow \infty $, we first get  the expansion of energy functional $ I(W_{r,h,\Lambda} )$
\begin{align*}
	F_1  (r,h,\Lambda) : = I(W_{r,h,\Lambda} )
	& \,=\,  k  A_1 -\frac{k}{\Lambda^{N-2}} \Big[\frac{B_4 k^{N-2}}{(r \sqrt{1-h^2})^{N-2}}\,+\, \frac{B_5 k}{r^{N-2} h^{N-3} \sqrt {1-h^2}}\Big]
	\\[2mm]
	& \quad+k \Big[\frac{A_2}{\Lambda^{m}  k^{\frac{(N-2)m}{N-2-m}}}
	+\frac{A_3}{\Lambda^{m-2}  k^{\frac{(N-2)m}{N-2-m}}}
	({\bf r}-r)^2\Big]
	+k  O \Big(\frac{1}{k^{\frac{(N-2)m}{N-2-m}+\sigma}}\Big),
\end{align*}
where $ A_i$ for $i=1,2,3$ and  $ B_j$ for $ j=4,5 $ are constants.
Denote
\begin{align*}
	\mathcal {G}(h) &:=
	\frac{B_4 k^{N-2}}{(\sqrt{1-h^2})^{N-2}}\,+\, \frac{B_5 k}{h^{N-3} \sqrt {1-h^2}}.
\end{align*}
Let  $h $ be the solution of $ \partial _h  \mathcal {G}(h)= 0$, then
$$ h=  \frac{B'}{k^{\frac{N-3}{N-1}}} \big(1+o(1)\big),  \quad {\mbox {as}} \quad k \to \infty$$
for some  $B' >0 $.
If $ r \sim k^{\frac{N-2}{N-2-m}}, h \sim \frac{B'}{k^{\frac{N-3}{N-1}}} $, then
\begin{equation*}
	\frac{B_5 k}{r^{N-2} h^{N-3} \sqrt {1-h^2}} =  \frac{\tilde B}{k^{\frac{(N-2)m}{N-2-m}+\frac{2(N-3)}{N-1}}}(1+o(1)) , \quad {\mbox {as}} \quad k \to \infty
\end{equation*}
for some constant $\tilde B$.

However, we now find  that  the term $O \Big(\frac{1}{k^{\frac{(N-2)m}{N-2-m}+\sigma}}\Big) $ in the expansion of $F_1 (r,h, \Lambda)$   competes with the term  $\frac{B_5 k}{r^{N-2} h^{N-3} \sqrt {1-h^2}}$, making it impossible to identify a critical point for  $  F_1  (r,h,\Lambda)$.  In reality though the remainder $O \Big(\frac{1}{k^{\frac{(N-2)m}{N-2-m}+\sigma}}\Big) $ can be estimated in a more   accurate way (see Proposition \ref{func}) under our assumptions   ${\bf H} $.

We need to expand the full energy $ F(r,h,\Lambda) = I(W_{r, h, \Lambda}+\phi_{r, h, \Lambda})$. We need a strong control on  the size of $\phi_{r, h, \Lambda}$ in order not to destroy the critical point structure of $F_1 (r,h, \Lambda)$ and to ensure the qualitative properties of the solutions as stated in Theorem \ref{main1}. This is another delicate step of our construction, where we make full use of the assumption  ${\bf H}$ on $K$.

\medskip
{\bf Structure of the paper.} The remaining part of this paper is devoted to the proof of Theorem \ref{main1}, which will be organized as follows:

 \begin{itemize}
\item[1.]
In Section \ref{sec3}, we will establish the linearized theory for the linearized projected problem. We will give estimates for the error terms in this Section.

\item[2.]
In Section \ref{sec4}, we shall proof Theorem \ref{main1} by showing there exists a critical point of reduction function  $  F(r,h,\Lambda) $.

\item[3.]
Some tedious computations and some useful Lemmas will be given in Appendices \ref{appendixA}-\ref{appendixB}.
 \end{itemize}


\medskip

{\bf Notation and preliminary results.} For the readers' convenience, we will provide a collection of notation. Throughout this paper, we employ  $C,  C_j  $ to denote certain constants and  $ \sigma, \tau,  \sigma_j  $ to denote some small constants or functions.  We also note that $ \delta_{ij} $  is Kronecker delta function:
\[ \delta_{ij}= \begin{cases} 1,  \quad \text{if} ~ i= j, \\[2mm]
0,  \quad \text{if} ~ i \neq j.
\end{cases} \]
Furthermore, we  also  employ the common notation  by writing  $O(f(r,h)), o(f(r,h))  $ for the functions which satisfy
\begin{equation*}
\text{if} \quad g(r,h) \in O(f(r,h)) \quad \text{then}\quad {\lim_{k \to+\infty}} \Big|\, \frac{g(r,h)}{f(r,h)} \, \Big|\leq C<+\infty,
\end{equation*}
and
\begin{equation*}
\text{if} \quad g(r,h) \in o(f(r,h)) \quad \text{then}\quad {\lim_{k \to+\infty}} \frac{g(r,h)}{f(r,h)}=0.
\end{equation*}

\section{Finite dimensional reduction}  \label{sec3}

For  $j= 1,\cdots, k $, we divide $\mathbb{R}^N$ into $k$ parts:
\begin{align*}
\Omega_j := &\Big\{y=(y_1, y_2, y_3, y'') \in \mathbb{R}^3 \times \mathbb{R}^{N-3}:
\nonumber\\[2mm]
& \qquad\Big\langle \frac{(y_1, y_2)}{|(y_1, y_2)|},  \Big(\cos{\frac{2(j-1) \pi}{k}}, \sin{\frac{2(j-1) \pi}{k}}\Big)  \Big\rangle_{\mathbb{R}^2}\geq \cos{\frac \pi k}\Big\}.
\end{align*}
where $ \langle , \rangle_{\mathbb{R}^2}$ denote the dot product in $\mathbb{R}^2$. For  $\Omega_j$, we further divide it into two parts:
\begin{align*}
\Omega_j^+= & \Big\{y:  y=(y_1, y_2, y_3, y'')  \in  \Omega_j, y_3\geq0 \Big\},
\\[2mm]
\Omega_j^-= & \Big\{y: y=(y_1, y_2, y_3, y'')  \in  \Omega_j, y_3<0 \Big\}.
\end{align*}
We can know that
$$\mathbb{R}^N= \cup_{j=1}^k  \Omega_j,  \quad \Omega_j=  \Omega_j^+\cup  \Omega_j^-$$
and
$$ \Omega_j \cap  \Omega_i=\emptyset,  \quad  \Omega_j^+\cap  \Omega_j^-=\emptyset, \qquad \text{if}\quad i\neq j.$$

\medskip
We consider the following linearized problem
\begin{equation}\label{lin}
\begin{cases}
-\Delta {\phi}-(2^*-1)K\big(\frac{|y|}{{\bf r}}\big)
W_{r,h,\Lambda}^{2^*-2}\phi=f+\sum\limits_{i=1}^k\sum\limits_{\ell=1}^3 \Big({c}_\ell U_{\overline{x}_i,\Lambda}^{2^*-2}\overline{\mathbb{Z}}_{\ell i}+{c}_\ell U_{\underline{x}_i,\Lambda}^{2^*-2}\underline{\mathbb{Z}}_{\ell i}\Big)
\;\;
\text{in}\;
\R^N,\\[2mm]
\phi\in \mathbb{E},
\end{cases}
\end{equation}
for some constants $c_{\ell} $.

\medskip
Coming back to equation \eqref{criticalequation}, we recall that the functions
\begin{equation}\label{Zj}
	Z_i(y) :={\partial U \over \partial y_i}(y), \quad i=1, \ldots , N, \quad
	Z_{N+1}(y) :={N-2 \over 2} U(y)+y\cdot \nabla U(y).
\end{equation}
belong to the null space of the linearized problem associated to \eqref{criticalequation} around an Aubin-Talenti bubble, namely they solve
\begin{equation}\label{linearized}
	\Delta \phi+(2^*-1) U^{2^*-2} \phi=0,~ \text{in}~ \mathbb{R}^N, \quad \phi \in D^{1,2}(\mathbb{R}^N).
\end{equation}
It is known \cite{Re} that these functions  span the set of the solutions to \eqref{linearized}. This fact will be used in
 the following crucial lemma which concerns the linearized problem \eqref{lin}.

\medskip
\begin{lem}\label{lem2.1}
 Suppose that $\phi_{k}$ solves \eqref{lin} for $f=f_{k}$. If  $\|f_{k}\|_{**}$ tends to zero as  $k$ tends to infinity, so does  $\|\phi_{k}\|_{*}$.
 \end{lem}

The norms $\| \cdot \|_*$ and $\| \cdot \|_{**}$ are defined respectively in \eqref{star} and \eqref{starstar}.

\medskip

 \begin{proof}
 We prove the Lemma  by contradiction. Suppose that there exists a sequence of  $(r_k, h_k, \Lambda_k)\in {{\mathscr S}_k}$, and for $\phi_k$ satisfies \eqref{lin} with  $f=f_k, r= r_k, h= h_k,  \Lambda= \Lambda_k$, with  $\Vert f _{k}\Vert_{**}\to 0 $, and  $\|\phi_k\|_*\ge c'>0$. Without loss of generality, we can assume that  $\|\phi_k\|_*=1 $. For convenience, we drop the subscript  $k $.

From \eqref{lin}, we know that
\begin{equation*}\begin{split}
\phi(y)
\,=\,&(2^*-1)\int_{\R^N}\frac1{|z-y|^{N-2}}K\Big(\frac{|z|}{{\bf r}}\Big)
W_{r,h,\Lambda}^{2^*-2}\phi(z)\,{\mathrm d}z
+\int_{\R^N}\frac1{|z-y|^{N-2}}\,  f(z) {\mathrm d}z
\\[2mm]
&+\int_{\R^N}\frac1{|z-y|^{N-2}}\, \sum_{j=1}^k\sum_{\ell=1}^3
\Big(\, {c_\ell}U_{\overline{x}_j,\Lambda}^{2^*-2}\overline{\mathbb{Z}}_{\ell j}+{c_\ell}U_{\underline{x}_j,\Lambda}^{2^*-2}\underline{\mathbb{Z}}_{\ell j}\,\Big) {\mathrm d}z
\\[2mm]
:=\,&M_1\,+\,M_2\,+\,M_3.
\end{split}
\end{equation*}
For the first term  $M_1 $, we make use of Lemma \ref{laa3}, so that
\begin{align*}
M_1 &\leq C\|{\phi}\|_* \, \int_{\R^N}\frac{K\big(\frac{|z|}{{\bf r}}\big)}{|z-y|^{N-2}}
W_{r,h,\Lambda}^{2^*-2} \Big(\sum_{j=1}^k \Big[\frac1{(1+|z-\overline{x}_j|)^{\frac{N-2}{2}+\tau
}}+\frac1{(1+|z-\underline{x}_j|)^{\frac{N-2}{2}+\tau}}\Big]\Big) \, {\mathrm d}z
\\[2mm]
&\leq C\|{\phi}\|_* \, \sum_{j=1}^k \Big[\frac1{(1+|z-\overline{x}_j|)^{\frac{N-2}{2}+\tau+\sigma}}+\frac1{(1+|z-\underline{x}_j|)^{\frac{N-2}{2}+\tau+\sigma}}\Big]\nonumber.
\end{align*}

For the second term  $M_2 $, we make use of Lemma \ref{lemb2}, so that
\begin{align*}
M_2& \le C\|f\|_{**}\int_{\R^N}\frac1{|z-y|^{N-2}}\sum_{j=1}^k \Big[\frac1{(1+|z-\overline{x}_j|)^{\frac{N+2}2+\tau}}+\frac1{(1+|z-\underline{x}_j|)^{\frac{N+2}2+\tau}}\Big]\,{\mathrm d}z
\\[2mm]
& \le C\|f\|_{**} \sum_{j=1}^k \Big[\frac1{(1+|y-\overline{x}_j|)^{\frac{N-2}{2}+\tau}}+\frac1{(1+|y-\underline{x}_j|)^{\frac{N-2}2+\tau}}\Big] \nonumber.
\end{align*}

In order to estimate the term  $M_3 $, we will first give the estimates of  $ \overline{\mathbb{Z}}_{1j} $ and  $ \underline{\mathbb{Z}}_{1j} $
\begin{equation}\label{estimatekernel}
\begin{split}
|\overline{\mathbb{Z}}_{1j}|\le  \frac{C}{(1+|y-\overline{x}_j|)^{N-2}},
\quad|\overline{\mathbb{Z}}_{2j}|\le  \frac{C r}{(1+|y-\overline{x}_j|)^{N-2}},
\quad|\overline{\mathbb{Z}}_{3j}|\leq  \frac{C}{(1+|y-\overline{x}_j|)^{N-2}},
 \\[2mm]
|\underline{\mathbb{Z}}_{1j}|\le  \frac{C}{(1+|y-\underline{x}_j|)^{N-2}},
 \quad|\underline{\mathbb{Z}}_{2j}|\le  \frac{C r}{(1+|y-\underline{x}_j|)^{N-2}},
 \quad|\underline{\mathbb{Z}}_{3j}|\leq  \frac{C}{(1+|y-\underline{x}_j|)^{N-2}}.
\end{split}
\end{equation}
Combining estimates \eqref{estimatekernel} and Lemma \ref{lemb2},  we have
\begin{align*}
\sum_{j=1}^k\, \int_{\R^N}\frac1{|z-y|^{N-2}}\, U_{\overline{x}_j,\Lambda}^{2^*-2}\overline{\mathbb{Z}}_{\ell j} \, {\mathrm d}z
 & \le C \sum_{j=1}^k\, \int_{\R^N}\frac1{|z-y|^{N-2}}\, \frac{(1+r \delta_{\ell 2})}{(1+|z-\overline{x}_j|)^{N+2}} \, {\mathrm d}z
\\[2mm]
& \le C\sum_{j=1}^k\, \frac{(1+r\delta_{\ell 2})}{(1+|y-\overline{x}_j|)^{\frac{N-2}{2}+\tau}}, \quad \text{for} ~\ell= 1, 2, 3,\nonumber
\end{align*}
where $\delta_{\ell 2} =0$ if $\ell \not= 2$, $\delta_{\ell 2}=1$ if $\ell =2$.
 Similarly, we have
\begin{align*}
 \sum_{j=1}^k\, \int_{\R^N}\frac1{|z-y|^{N-2}}\,
  U_{\underline{x}_j,\Lambda}^{2^*-2}\underline{\mathbb{Z}}_{\ell j}\,{\mathrm d}z
 \le C  \sum_{j=1}^k\,  \frac{(1+r \, \delta_{\ell 2})}{(1+|y-\underline{x}_j|)^{\frac{N-2}{2}+\tau}}, \quad \text{for} ~\ell= 1, 2, 3.
\end{align*}

\medskip
 Next, we will give the estimates of  ${c_\ell}, \ell= 1,2,3 $. Multiply both sides of \eqref{lin} by  $\overline{\mathbb{Z}}_{q 1}, q=1,2,3 $, then we obtain that
 \begin{equation}\label{equationofcl}
 \begin{split}
 & \int_{\R^N}  \Big[-\Delta {\phi}-(2^*-1)K\Big(\frac{|y|}{{\bf r}}\Big)
W_{r,h,\Lambda}^{2^*-2}{\phi}\Big] \overline{\mathbb{Z}}_{q 1}
\\[2mm]
&= \int_{\R^N} f \,\overline{\mathbb{Z}}_{q 1}+\sum_{j=1}^k\sum_{\ell=1}^3  \int_{\R^N}  \Big(\, {c_\ell}U_{\overline{x}_j,\Lambda}^{2^*-2}\overline{\mathbb{Z}}_{\ell j}+{c_\ell}U_{\underline{x}_j,\Lambda}^{2^*-2}\underline{\mathbb{Z}}_{\ell j}\,\Big)\,\overline{\mathbb{Z}}_{q 1}.
 \end{split}
 \end{equation}
 Using Lemma \ref{lemb1}, we can get
\begin{align*}
\int_{\R^N} f \, \overline{\mathbb{Z}}_{q 1}
& \le C\|f\|_{**}\,\sum_{j=1}^k  \int_{\R^N}  \frac{1+r \, \delta_{\ell 2}}{(1+|y-\overline{x}_1|)^{N-2}} \Big[\frac1{(1+|y-\overline{x}_j|)^{\frac{N+2}2+\tau}}
+\frac1{(1+|y-\underline{x}_j|)^{\frac{N+2}2+\tau}}\Big]\,  \nonumber
\\[2mm]
& \le C(1+r \, \delta_{\ell 2}) \|f_k\|_{**}.
\end{align*}

 The discussion on the left side of \eqref{equationofcl} may be more tricky, in fact, we have
\begin{align*}
 & \int_{\R^N}  \Big[-\Delta {\phi}-(2^*-1)K\Big(\frac{|y|}{{\bf r}}\Big)
W_{r,h,\Lambda}^{2^*-2}{\phi}\Big] \overline{\mathbb{Z}}_{q 1} \nonumber
\\[2mm]
&= \int_{\R^N} \Big[-\Delta \overline{\mathbb{Z}}_{q 1} -(2^*-1)K\Big(\frac{|y|}{{\bf r}}\Big)
W_{r,h,\Lambda}^{2^*-2} \overline{\mathbb{Z}}_{q 1}\Big] {\phi}  \nonumber
\\[2mm]
&=(2^*-1) \int_{\R^N} \Big[1-K\Big(\frac{|y|}{{\bf r}}\Big)\Big] W_{r,h,\Lambda}^{2^*-2} \overline{\mathbb{Z}}_{q 1}{\phi}+\Big(U_{\overline{x}_1,\Lambda}^{2^*-2}-W_{r,h,\Lambda}^{2^*-2}\Big)  \overline{\mathbb{Z}}_{q 1}{\phi}
\\[2mm]
&:=J_1+J_2 \nonumber .
\end{align*}
Using the property of  $ K(s)  $, similar to the proof of Lemma \ref{laa3}, we can get
\begin{align*}
J_1\le& C\|{\phi}\|_{*}  \int_{\R^N}\Big|1-K\Big(\frac{|y|}{{\bf r}}\Big) \Big| \, W_{r,h,\Lambda}^{2^*-2}  \overline{\mathbb{Z}}_{q 1}  \sum_{j=1}^k \Big[\frac1{(1+|y-\overline{x}_j|)^{\frac{N-2}{2}+\tau}}+\frac1{(1+|y-\underline{x}_j|)^{\frac{N-2}{2}+\tau}}\Big]
\\[2mm]
= & C\|{\phi}\|_{*} \int\limits_{||y|-{\bf r}|\le \sqrt{\bf r}}\Big|1-K\Big(\frac{|y|}{{\bf r}}\Big)\Big| \, W_{r,h,\Lambda}^{2^*-2} \overline{\mathbb{Z}}_{q 1}  \sum_{j=1}^k \Big[\frac1{(1+|y-\overline{x}_j|)^{\frac{N-2}{2}+\tau}}+\frac1{(1+|y-\underline{x}_j|)^{\frac{N-2}{2}+\tau}}\Big]
\\[2mm]
&+C\|{\phi}\|_{*} \int\limits_{||y|-{\bf r}|\ge \sqrt{\bf r}}\Big|1-K\Big(\frac{|y|}{{\bf r}}\Big)\Big| \, W_{r,h,\Lambda}^{2^*-2} \overline{\mathbb{Z}}_{q 1}  \sum_{j=1}^k \Big[\frac1{(1+|y-\overline{x}_j|)^{\frac{N-2}{2}+\tau}}+\frac1{(1+|y-\underline{x}_j|)^{\frac{N-2}{2}+\tau}}\Big]
\\[2mm]
\le & \frac{C}{\sqrt{\bf r}}\int_{\R^N}
W_{r, h, \Lambda}^{2^*-2}(y) \frac{1+r\, \delta_{\ell 2}}{(1+|y-\overline{x}_1|)^{N-2}}\sum_{j=1}^k \frac1{(1+|y-\underline{x}_j|)^{\frac{N-2}{2}+\tau}}\,
\\[2mm]
 &+\frac{C}{{\bf r}^\sigma}\int_{\R^N}
W_{r,h,\Lambda}^{2^*-2}(y) \frac{1+r\, \delta_{\ell 2}}{(1+|y-\overline{x}_1|)^{N-2}}\sum_{j=1}^k \frac1{(1+|y-\overline{x}_j|)^{\frac{N-2}{2}+\tau-2\sigma}}\,
\leq\frac{C}{{\bf r}^\sigma}(1+r\, \delta_{\ell 2}).
\end{align*}

For $J_2$, it is easy to derive that
\begin{align*}
J_2
\,\le\,& \int_{\R^N} \Big|U_{\overline{x}_1,\Lambda}^{2^*-2}-W_{r,h,\Lambda}^{2^*-2} \Big|\frac{1+r \, \delta_{\ell 2}}{(1+|y-\overline{x}_1|)^{N-2}}
\\[2mm]
& \quad \quad \times  \sum_{j=1}^k \Big[\frac1{(1+|y-\overline{x}_j|)^{\frac{N-2}{2}+\tau}}+\frac1{(1+|y-\underline{x}_j|)^{\frac{N-2}{2}+\tau}}\Big]
\\[2mm]
\,\le\,& \frac{C}{{\bf r}^\sigma}(1+r\, \delta_{\ell 2}).
\end{align*}
Then, we get
\begin{align*}
\int_{\R^N}  \Big[-\Delta {\phi}-(2^*-1)K\Big(\frac{|y|}{{\bf r}}\Big)
W_{r,h,\Lambda}^{2^*-2}\,{\phi}\Big] \overline{\mathbb{Z}}_{q 1} \le \frac{C}{{\bf r}^\sigma}(1+r \, \delta_{\ell 2}) \, \|{\phi}\|_{*}.
\end{align*}
On the other hand, there holds
\begin{align*}
\sum_{j=1}^k \int_{\R^N} \big(\, U_{\overline{x}_j,\Lambda}^{2^*-2}\overline{\mathbb{Z}}_{\ell j}+U_{\underline{x}_j,\Lambda}^{2^*-2}\underline{\mathbb{Z}}_{\ell j}\,\Big)\,\overline{\mathbb{Z}}_{q 1}
= \bar{c}_{\ell}\delta_{\ell q} (1+\delta_{q2} r^2)+o(1), \quad {\mbox {as}} \quad k \to \infty.
\end{align*}
 Note that
 \begin{eqnarray*}
\int_{\R^N}\,U_{\overline{x}_1,\Lambda}^{2^*-2}\overline{\mathbb{Z}}_{\ell 1}  \overline{\mathbb{Z}}_{q 1}=\left\{\begin{array}{rcl}  0,\qquad \text{if}\quad \ell \neq q,
&&\\[2mm]
\space\\[2mm]
 \bar{c}_q (1+\delta_{q2} r^2),\qquad \text{if} \quad\ell=q, \end{array}\right.
\end{eqnarray*}
 for some constant  $\bar{c}_q >0 $.
 Then we can get
 \begin{align} \label{estimate cl}
c_{\ell}=  \frac{1+r\delta_{\ell2}}{1+\delta_{\ell2} r^2} O\Bigl(\frac 1{{\bf r}^\sigma}\|{\phi}\|_{*}+\|f\|_{**} \Bigr)= o(1), \quad {\mbox {as}} \quad k \to \infty.
\end{align}
Then we have
\begin{equation}
\label{g8}
\begin{split}
|{\phi}| \le \Bigl(&\,\|f\|_{**}  \sum_{j=1}^k \Big[\frac1{(1+|y-\overline{x}_j|)^{\frac{N-2}{2}+\tau}}+\frac1{(1+|y-\underline{x}_j|)^{\frac{N-2}{2}+\tau}}\Big]
\\[2mm]
&+{\sum_{j=1}^k \Big[\frac1{(1+|y-\overline{x}_j|)^{\frac{N-2}{2}+\tau+\sigma}}+\frac1{(1+|y-\underline{x}_j|)^{\frac{N-2}{2}+\tau+\sigma}}\Big]}\, \Bigr).
\end{split}
\end{equation}
Combining this fact and  $\|{\phi}\|_{*}= 1 $, we have the following claim: \\[2mm]
 \textbf{Claim 1}:  There exist some positive constants  $ \bar{R}, \delta_1 $ such that
\begin{align} \label{bound1}
\|\phi\|_{L^\infty(B_{\bar{R}}(\overline{x}_l))}  \ge \delta_1> 0,
 \end{align}
for some  $ l\in \{1,2, \cdots, k\}  $.

\medskip
Since  $ \phi\in H_s $, we assume that  $l=1 $. By using local elliptic estimates and \eqref{g8}, we can get, up to subsequence,  $\tilde \phi(y)
= \phi(y-\overline{x}_1)  $  converge uniformly in
any compact set to a solution
 \begin{equation*}
-\Delta u-(2^*-1) U_{0,\Lambda}^{2^*-2} u=0,\quad \text{in}\;\R^N,
\end{equation*}
for some  $\Lambda\in [L_1,L_2] $.  Since  $ \phi  $ is even in  $y_d, d= 2, 4, \cdots,N $, we know that  $ u $ is also even in  $y_d, d= 2, 4, \cdots,N $.
Then we know that  $u $ must be a linear combination of the functions
\begin{equation*}
\frac{\partial U_{0,\Lambda}}{\partial y_1}, \qquad  \frac{\partial U_{0,\Lambda}}{\partial y_3}, \qquad y \cdot {\nabla U_{0,\Lambda}}+(N-2)  U_{0,\Lambda}.
\end{equation*}
From the assumptions
\[
\int_{{\mathbb{R}}^N} U_{\overline{x}_1, \Lambda}^{2^*-2}  \overline{\mathbb{Z}}_{\ell 1}\, \tilde \phi = 0 \quad \text{for} ~ \ell= 1,2,3,
\] we can get
\begin{align*}
 \sqrt{1-h^2} \, \int_{{\mathbb{R}}^N} U_{0, \Lambda}^{2^*-2}\,\frac{\partial U_{0,\Lambda}}{\partial y_1}\,\tilde \phi+h \, \int_{{\mathbb{R}}^N} U_{0, \Lambda}^{2^*-2}\,\frac{\partial U_{0,\Lambda}}{\partial y_3}\,\tilde \phi= 0,
 \\[2mm]
  \sqrt{1-h^2} \, \int_{{\mathbb{R}}^N} U_{0, \Lambda}^{2^*-2}\,\frac{\partial U_{0,\Lambda}}{\partial y_1}\,\tilde \phi+h \, \int_{{\mathbb{R}}^N} U_{0, \Lambda}^{2^*-2}\,\frac{\partial U_{0,\Lambda}}{\partial y_3}\,\tilde \phi= 0,
\end{align*}
and  \[ \int_{{\mathbb{R}}^N} U_{0, \Lambda}^{2^*-2}\,\Big[y \cdot {\nabla U_{0,\Lambda}}+(N-2)  U_{0,\Lambda} \Big]  \tilde \phi =0.
\]
By taking limit, we have
\begin{align*}
\int_{{\mathbb{R}}^N} U_{0, \Lambda}^{2^*-2}\,\frac{\partial U_{0,\Lambda}}{\partial y_1}\,u= \int_{{\mathbb{R}}^N} U_{0, \Lambda}^{2^*-2}\,\frac{\partial U_{0,\Lambda}}{\partial y_3}\,u= \int_{{\mathbb{R}}^N} U_{0, \Lambda}^{2^*-2}\,\Big[y \cdot {\nabla U_{0,\Lambda}}+(N-2)  U_{0,\Lambda} \Big]  u =0.
\end{align*}
So we have  $u=0 $.  This is a contradiction to \eqref{bound1}.
\end{proof}

\medskip
For the linearized problem \eqref{lin}, we have the following existence, uniqueness results. Furthermore, we can give the estimates of  $ \phi $ and  $c_\ell, \ell=1,2,3 $.

\medskip
\begin{proposition}\label{p1}
There exist  $k_0>0  $ and a constant  $C>0 $ such
that for all  $k\ge k_0 $ and all  $f\in L^{\infty}(\R^N) $, problem
 $(\ref{lin}) $ has a unique solution  $\phi\equiv {\bf L}_k(f) $. Besides,
\begin{equation}\label{Le}
\Vert \phi\Vert_*\leq C\|f\|_{**},\qquad
|c_{\ell}|\leq  \frac{C} {1+\delta_{\ell 2} r}\|f\|_{**}, \quad \ell=1,2,3.
\end{equation}
\end{proposition}

\begin{proof}
Recall  the definition of  $\mathbb{E}$ as in  \eqref{SpaceE}, we can rewrite problem \eqref{lin} in the form
\begin{equation} \label{Lin1}
-\Delta {\phi}=f+(2^*-1)K\Big(\frac{|y|}{{\bf r}}\Big)
W_{r,h,\Lambda}^{2^*-2} \phi \quad \text{for all} ~ \phi \in \mathbb{E},
\end{equation}
in the sense of distribution. Furthermore, by using Riesz's representation theorem,  equation \eqref{Lin1} can be rewritten in the operational form
\begin{equation}\label{lin1}
(\mathbb I- \mathbb{T}_k) \phi = \tilde{f}, \quad  \text{in} ~  \mathbb{E},
\end{equation}
where $\mathbb I$ is identity operator and $\mathbb{T}_k$ is a compact operator.
Fredholm's alternative yields that  problem \eqref{lin1} is uniquely solvable for any  $ \tilde f $ when the homogeneous equation
\begin{align}\label{homo}
(\mathbb I-\mathbb{T}_k) \phi = 0, \quad  \text{in} ~  \mathbb{E},
\end{align}
has only the trivial solution. Moreover, problem \eqref{homo} can be rewritten as following
\begin{equation}\label{lin0}
\begin{cases}
-\Delta {\phi}-(2^*-1)K\Big(\frac{|y|}{{\bf r}}\Big)
W_{r,h,\Lambda}^{2^*-2}\phi
= \sum\limits_{i=1}^k\sum\limits_{\ell=1}^3 \Big({c}_\ell U_{\overline{x}_i,\Lambda}^{2^*-2}\overline{\mathbb{Z}}_{\ell i}+{c}_\ell U_{\underline{x}_i,\Lambda}^{2^*-2}\underline{\mathbb{Z}}_{\ell i}\Big)\;\;\text{in}\;\R^N,
\\[2mm]
\phi\in \mathbb{E}.
\end{cases}
\end{equation}
 Suppose that \eqref{lin0} has nontrivial solution $\phi_k$ and satisfies $\| \phi_k\|_{*}=1$.  From Lemma \ref{lem2.1}, we know $\| \phi_k\|_{*}$ tends to zero as $ k\to+\infty$, which is a contradiction. Thus problem \eqref{homo} (or \eqref{lin0}) only has trivial solution. So we can get unique solvability for problem \eqref{lin}. Using Lemma \ref{lem2.1}, the estimates \eqref{Le} can be proved by a standard method.
\end{proof}

%

%
\medskip
We can rewrite problem \eqref{re} as following
\begin{equation}\label{re1}
\begin{cases}
-\Delta \phi
-(2^*-1)K\Big(\frac{|y|}{{\bf r}}\Big)
W_{r,h,\Lambda}^{2^*-2}\phi= {{\bf N}}(\phi)+{\bf l}_k
\\[2mm]
\qquad \qquad \qquad \qquad+\sum\limits_{j=1}^k\sum\limits_{\ell=1}^3
\Big(\, {c_\ell}U_{\overline{x}_j,\Lambda}^{2^*-2}\overline{\mathbb{Z}}_{\ell j}+{c_\ell}U_{\underline{x}_j,\Lambda}^{2^*-2}\underline{\mathbb{Z}}_{\ell j}\,\Big)\, \; \text{in}\;
\R^N,\\[2mm]
\phi \in  \mathbb{E},
\end{cases}
\end{equation}
where
\begin{equation*}
{{\bf N}}(\phi)=K\Big(\frac{|y|}{{\bf r}}\Big)\Big[\bigl(W_{r,h,\Lambda}+\phi \bigr)^{2^*-1}-W_{r,h,\Lambda}^{2^*-1}-(2^*-1)W_{r,h,\Lambda}^{2^*-2}\phi\Big],
\end{equation*}
and
\begin{equation*}
 {\bf l}_k=K\Big(\frac{|y|}{{\bf r}}\Big) W_{r,h,\Lambda}^{2^*-1}-\sum_{j=1}^k \Big(\, U_{\overline{x}_j,\Lambda}^{2^*-1}+U_{\underline{x}_j,\Lambda}^{2^*-1} \,\Big).
\end{equation*}
Next, we will use the Contraction Mapping Principle to show that problem \eqref{re1} has a unique solution in the set that  $\|\phi\|_* $ is small enough. Before that, we will give the estimate   of  ${{\bf N}}(\phi) $ and  $ {\bf l}_k $.

\medskip
\begin{lemma} \label{Lemma2.3}
Suppose  $ N\ge 5  $. There exists $C>0$ such that
\[
\|{\bf N}(\phi)\|_{**}\le C\|\phi\|_*^{\min\{2^*-1, 2\}},
\]
for all $\phi \in \mathbb{E}$.
\end{lemma}
\begin{proof}
The proof is similar to that of Lemma 2.4 in \cite{WY-10-JFA}. Here we omit it.
\end{proof}

\medskip
We next give the estimate of ${\bf l}_k$.
\begin{lemma} \label{Lemma2.4}
Suppose $K(|y|)$ satisfies ${\bf H}$ and $N\ge 5$,  $(r,h,\Lambda) \in{{\mathscr S}_k}$.
There exists $k_0 $ and $C>0$ such that for all $k \geq k_0$
\begin{align}  \label{estimateforlk}
\|\, {\bf l}_k \,\|_{**}  \le  C\max\Big\{ \frac1{k^{(\frac m{N-2-m})(
\frac{N+2}{2}-\frac{N-2-m}{N-2}- \epsilon_1 )}},  \frac1{k^{(\frac {N-2}{N-2-m}) \min\{m, \frac{m+3} 2\}}} \Big\},
\end{align}
where  $\epsilon_1$  is   small constant given  in  \eqref{tau and epsilon}.
\end{lemma}

\begin{proof}
We can rewrite ${\bf l}_k$ as
\[
\begin{split}
 {\bf l}_k
 \,=\,& K\Big(\frac{|y|}{{\bf r}}\Big) \Big[\, W_{r,h,\Lambda}^{2^*-1}-\sum_{j=1}^k \Big(\, U_{\overline{x}_j,\Lambda}^{2^*-1}+U_{\underline{x}_j,\Lambda}^{2^*-1} \,\Big)  \,\Big]
 \\[2mm]
 &+\sum_{j=1}^k  \Big[\ K\Big(\frac{|y|}{{\bf r}}\Big)-1  \,\Big]  \Big(\, U_{\overline{x}_j,\Lambda}^{2^*-1}+U_{\underline{x}_j,\Lambda}^{2^*-1} \,\Big)
 :=\, S_1+S_2.
\end{split}
\]

Assume that  $ y \in \Omega_1^+ $, then we get
\begin{align*}
S_1
&=  K\Big(\frac{|y|}{{\bf r}}\Big) \Big[\, \Big(\, \sum_{j=1}^k  U_{\overline{x}_j,\Lambda}+U_{\underline{x}_j,\Lambda}  \,\Big)^{2^*-1}-\sum_{j=1}^k \Big(\, U_{\overline{x}_j,\Lambda}^{2^*-1}+U_{\underline{x}_j,\Lambda}^{2^*-1} \,\Big)  \,\Big]
\nonumber \\[2mm]
& \leq  C K\Big(\frac{|y|}{{\bf r}}\Big)  \Big[U_{\overline{x}_1,\Lambda}^{2^*-2}\, \Big(\sum_{j=2}^k  U_{\overline{x}_j,\Lambda}+\sum_{j=1}^k U_{\underline{x}_j,\Lambda}\Big)+\Big(\sum_{j=2}^k  U_{\overline{x}_j,\Lambda}+\sum_{j=1}^k U_{\underline{x}_j,\Lambda}\Big)^{2^*-1}\Big].
\end{align*}
 Thus, we have
\begin{align*}
S_1 \,\le\, & C\frac1{(1+|y-\overline{x}_1|)^4} \sum_{j=2}^k \frac1{(1+|y-\overline{x}_j|)^{N-2}}
+C\frac1{(1+|y-\overline{x}_1|)^4} \sum_{j=1}^k \frac1{(1+|y-\underline{x}_j|)^{N-2}}
\\[2mm]
&+C\Bigl(\sum_{j=2}^k \frac1{(1+|y-\overline{x}_j|)^{N-2}}\Bigr)^{2^*-1}
:=\,S_{11}+S_{12}+S_{13}.
\end{align*}

 We first consider the case $N=5$. It is easy to get that
 \begin{align}\label{S11N=5}
S_{11}\big|_{N=5}  \,\le\, & C\frac1{(1+|y-\overline{x}_1|)^{\frac 72+ \tau}} \sum_{j=2}^k \frac1{|\overline{x}_j-\overline{x}_1|^{3}}  \nonumber
\\[2mm]
 \,\le\, & C\frac1{(1+|y-\overline{x}_1|)^{\frac 72+ \tau}} \Big(\frac k{\bf r} \Big)^3.
 \end{align}
 When $N\ge 6$, similar to the proof of Lemma \ref{b.0}, for any  $  1 <  \alpha_1  < N-2 $, we have
\begin{align*}
\sum_{j=2}^k \frac1{(1+|y-\overline{x}_j|)^{N-2}}
\le \frac{C} {(1+|y-\overline{x}_1|)^{N-2-\alpha_1}}\,\sum_{j=2}^k \, \frac1{|\overline{x}_j-\overline{x}_1|^{\alpha_1}}.
\end{align*}
Since  $ \tau  \in ( \frac{N-2-m}{N-2},  \frac{N-2-m}{N-2}+ \epsilon_1 ) $, we can choose $\alpha_1$ satisfies
\begin{equation*}
\frac{N+2}{2}-\frac{N-2-m}{N-2}- \epsilon_1<\alpha_1   =\frac{N+2}{2}-\tau<N-2.
\end{equation*}
Then
\begin{align} \label{S11n>5}
 S_{11}\big|_{N\ge 6}
& \le  \frac{C} {(1+|y-\overline{x}_1|)^{N+2-\alpha_1}}\,\sum_{j=2}^k \, \frac1{|\overline{x}_j-\overline{x}_1|^{\alpha_1}} \nonumber
\\[2mm]
& \le  \frac{C} {(1+|y-\overline{x}_1|)^{N+2-\alpha_1}} \, \Big(\frac{k}{{\bf r}\,\sqrt{1-h^2} }\Big)^{\alpha_1}\nonumber
\\[2mm]
&\le C\frac1{(1+|y-\overline{x}_1|)^{\frac{N+2}2+\tau}}\Big(\frac k{\bf r} \Big)^{
\frac{N+2}{2}-\frac{N-2-m}{N-2}- \epsilon_1 }.
\end{align}
Then combining \eqref{S11N=5} and \eqref{S11n>5}, we can get
 \begin{eqnarray}\label{S11}
\|S_{11}\|_{**}\le   \left\{\begin{array}{rcl}    C\Big(\frac k{\bf r} \Big)^{
\frac{N+2}{2}-\frac{N-2-m}{N-2}- \epsilon_1 },\qquad \text{if}\quad N \ge  6,
&&\\[2mm]
\space\\[2mm]
C\Big(\frac k{\bf r} \Big)^3,\qquad \text{if} \quad N=5.  \end{array}\right.
\end{eqnarray}

For $S_{12}$, we can rewrite it as following
\begin{align*}
 S_{12} &= C\frac1{(1+|y-\overline{x}_1|)^4}\Big[\frac1{(1+|y-\underline{x}_1|)^{N-2}}+\sum_{j=2}^k \frac1{(1+|y-\underline{x}_j|)^{N-2}}\Big]
\\[2mm]
&\le C\frac1{(1+|y-\overline{x}_1|)^4}\Big[\frac1{(1+|y-\underline{x}_1|)^{N-2}}+\sum_{j=2}^k \frac1{(1+|y-\bar{x}_j|)^{N-2}}\Big].
\end{align*}
Similarly  to \eqref{S11N=5},
we can obtain
\begin{align*}
S_{12}\big|_{N=5}
 \,\le\,  C\frac1{(1+|y-\overline{x}_1|)^{\frac 72+ \tau}} \Big(\frac k{\bf r} \Big)^3.
 \end{align*}
For $N\geq 6$ and the same  $\alpha_1$ as in \eqref{S11}, it is easy to derive that
\begin{align*}
 &\frac1{(1+|y-\overline{x}_1|)^4}\frac1{(1+|y-\underline{x}_1|)^{N-2}}
\\[2mm]
& \le \Big[\,\frac1{(1+|y-\overline{x}_1|)^{N+2-\alpha_1}}+\frac1{(1+|y-\underline{x}_1|)^{N+2-\alpha_1}}\,\Big]\frac1{|\underline{x}_1-\overline{x}_1|^{
\alpha_1}}
\\[2mm]
& \le  \frac{C} {(1+|y-\overline{x}_1|)^{N+2-\alpha_1}}\,\frac 1 {(hr)^{\alpha_1}}
\\[2mm]
& \le C\frac1{(1+|y-\overline{x}_1|)^{\frac{N+2}2+\tau}}\Big(\frac k{\bf r} \Big)^{
\frac{N+2}{2}-\frac{N-2-m}{N-2}- \epsilon_1 },
\end{align*}
where we have used the fact $ h r >  C \frac{\bf r} k$.
Thus, we can obtain that
 \begin{eqnarray}\label{S12}
 \|S_{12}\|_{**} \le   \left\{\begin{array}{rcl}  C\Big(\frac k{\bf r} \Big)^{
\frac{N+2}{2}-\frac{N-2-m}{N-2}- \epsilon_1 },\qquad \text{if}\quad N \ge  6,
&&\\[2mm]
\space\\[2mm]
C \Big(\frac k{\bf r} \Big)^3,\qquad \text{if} \quad N=5.  \end{array}\right.
\end{eqnarray}

Next, we consider  $S_{13}$. For  $ y \in \Omega_1^+$,
\begin{align*}
\sum_{j=2}^k\, \frac1{(1+|y-\overline{x}_j|)^{N-2}}
& \le  \sum_{j=2}^k \frac1{(1+|y-\overline{x}_1|)^{\frac{N-2}{2}}} \, \frac1{(1+|y-\overline{x}_j|)^{\frac{N-2}{2}}}
\nonumber\\[2mm]
& \le  \sum_{j=2}^k\,\frac{C}{|\overline{x}_j-\overline{x}_1|^{\frac{N-2}{2}-\frac{N-2}{N+2}\tau}}
\frac 1{(1+|y-\overline{x}_1|)^{\frac{N-2}{2}+\frac{N-2}{N+2}\tau}}
\\[2mm]
& \le C\Big(\,\frac k {{\bf r}\sqrt{1-h^2}}\,\Big)^{\frac{N-2}{2}-\frac{N-2}{N+2}\tau} \frac 1{(1+|y-\overline{x}_1|)^{\frac{N-2}{2}+\frac{N-2}{N+2}\tau}}.
\nonumber
\end{align*}
Thus we have
\begin{equation*}
\begin{split}
S_{13}
& \le \, \Big(\, \frac k {{\bf r}\sqrt{1-h^2}}  \,\Big)^{\frac{N+2}{2}-\tau} \frac{C}{(1+|y-\overline{x}_1|)^{\frac{N+2}{2}+\tau}}
 \\[2mm]
& \le\frac{C}{(1+|y-\overline{x}_1|)^{\frac{N+2}{2}+\tau}}  \Big(\frac k{\bf r} \Big)^{
\frac{N+2}{2}-\frac{N-2-m}{N-2}- \epsilon_1}.
\end{split}
\end{equation*}
Since  $\big(\frac{N+2}{2}-\frac{N-2-m}{N-2}- \epsilon_1\big)\big|_{N=5}   >3$ for $m\in [2, 3)$,  then we have
 \begin{eqnarray}\label{S13}
\| S_{13}\|_{**} \le   \left\{\begin{array}{rcl}  C\Big(\frac k{\bf r} \Big)^{
\frac{N+2}{2}-\frac{N-2-m}{N-2}- \epsilon_1 },\qquad \text{if}\quad N \ge  6,
&&\\[2mm]
\space\\[2mm]
C \Big(\frac k{\bf r} \Big)^3,\qquad \text{if} \quad N=5.  \end{array}\right.
\end{eqnarray}
 Combining \eqref{S11}, \eqref{S12}, \eqref{S13}, we obtain
 \begin{eqnarray}\label{S1}
\|S_1\|_{**} \le   \left\{\begin{array}{rcl}  C \Big(\frac k{\bf r} \Big)^{
\frac{N+2}{2}-\frac{N-2-m}{N-2}- \epsilon_1 },\qquad \text{if}\quad N \ge  6,
&&\\[2mm]
\space\\[2mm]
C \Big(\frac k{\bf r} \Big)^3,\qquad \text{if} \quad N=5.  \end{array}\right.
\end{eqnarray}

We now consider the estimate of $S_2$.  For  $ y\in  \Omega_1^+ $, we have
\begin{align*}
S_2 \,\le \,& 2\sum_{j=1}^k  \Big[\ K\Big(\frac{|y|}{{\bf r}}\Big)-1  \,\Big]  \, U_{\overline{x}_j,\Lambda}^{2^*-1}
\\[2mm]
\,=\,& 2\,U_{\overline{x}_1,\Lambda}^{2^*-1} \,\Big[K \Big(\frac{|y|}{{\bf r}}\Big)-1\Big]
+2\,\sum_{j=2}^k \, U_{\overline{x}_j,\Lambda}^{2^*-1} \, \Big[K\Big(\frac{|y|}{{\bf r}}\Big)-1\Big]
\\[2mm]
:=\,& S_{21}+S_{22}.
\end{align*}
$\bullet $ If  $|\frac{|y|}{\bf r}-1|\ge \delta_1,  $ where  $ \delta> \delta_1> 0  $, then
\begin{align*}
|y-\overline{x}_1| \ge \big||y |-{\bf r}\big|\,-\,\big|{\bf r} -|\overline{x}_1|\big| \ge \frac{1}{2} \delta_1 {\bf r}.
\end{align*}
As a result, we get
\begin{align*}
U_{\overline{x}_1,\Lambda}^{2^*-1} \,\Big[K\Big(\frac{|y|}{{\bf r}}\Big)-1\Big]
& \le \frac{C}{\big(1+|y-\overline{x}_{1}|\big)^{\frac{N+2} 2+\tau}}  \frac{1}{{\bf r}^{\frac{N+2} 2-\tau}}
\\[2mm]
& \le \frac{C}{\big(1+|y-\overline{x}_{1}|\big)^{\frac{N+2} 2+\tau}}  \Big(\frac k{\bf r} \Big)^{
\frac{N+2}{2}-\frac{N-2-m}{N-2}- \epsilon_1 }.
\end{align*}
$\bullet $ If $|\frac{|y|}{\bf r}-1|\le \delta_1,$ then
\begin{align*}
\Big[K\Big(\frac{|y|}{{\bf r}}\Big)-1\Big]  \le& C\Big|\frac{|y|}{\bf r}-1\Big|^{m} = \frac{C} {{\bf r}^{m}}||y|-{\bf r}|^{m}
\\[2mm]
\leq&\frac{C} {{\bf r}^{m}}\Big[\big||y|-|\overline{x}_1|\big|^{m}
\,+\,\big||\overline{x}_1|-{\bf r}\big|^{m}\Big]
\\[2mm]
\leq&\frac{C} {{\bf r}^{m}}\Big[\big||y|-|\overline{x}_1|\big|^{m}
 \,+\,\frac{1}{k^{{\bar\theta m}}}\Big].
\end{align*}
 Thus, we can get, if $m>3$,
\begin{align*}
 U_{\overline{x}_1,\Lambda}^{2^*-1} \,\Big[K\Big(\frac{|y|}{{\bf r}}\Big)-1\Big]
 &  \le \frac{C} {{\bf r}^{m}}\Big[\big||y|-|\overline{x}_1|\big|^{m}
 \,+\,\frac{1}{k^{{\bar\theta m}}}\Big] \frac{C}{\big(1+|y-\overline{x}_{1}|\big)^{N+2}}
 \nonumber\\[2mm]
 &\leq\frac{C} {{\bf r}^{\frac{m+3} 2 }}\Big[\frac{\big||y|-|\overline{x}_1|\big|^{\frac{m+3} 2 }}{\big(1+|y-\overline{x}_{1}|\big)^{N+2}}
 \,+\,\frac{1} {{\bf r}^{\frac{m-3} 2 }}\frac{1}{k^{{\bar\theta m}}}\frac{1}{\big(1+|y-\overline{x}_{1}|\big)^{N+2}} \Big]
\\
 & \le  \frac{C}{{\bf r}^{\frac{m+3} 2 }} \Big[ \frac{1}{\big(1+|y-\overline{x}_{1}|\big)^{\frac{N+2} 2+\tau}}  \frac{1}{\big(1+|y-\overline{x}_{1}|\big)^{\frac{N+2} 2- \tau-\frac{m+3}{2}}}
 +\frac{1}{\big(1+|y-\overline{x}_{1}|\big)^{N+2}}\Big]
 \\[2mm]
 & \le  \frac{1}{{\bf r}^{\frac{m+3} 2 }}  \frac{C}{\big(1+|y-\overline{x}_{1}|\big)^{\frac{N+2} 2+\tau}},
\end{align*}
the last   inequality holds due to  $\frac{N+2} 2- \tau-\frac{m+3}{2}>0.$

\medskip
On  the other hand,  if $m\le 3$, we have
\begin{align*}
 U_{\overline{x}_1,\Lambda}^{2^*-1} \,\Big[K\Big(\frac{|y|}{{\bf r}}\Big)-1\Big]
 &  \le \frac{C} {{\bf r}^{m}}\Big[\big||y|-|\overline{x}_1|\big|^{m}
 \,+\,\frac{1}{k^{{\bar\theta m}}}\Big] \frac{C}{\big(1+|y-\overline{x}_{1}|\big)^{N+2}}  \nonumber
 \\
 & \le  \frac{C}{{\bf r}^{m}} \Big[ \frac{1}{\big(1+|y-\overline{x}_{1}|\big)^{\frac{N+2} 2+\tau}}  \frac{1}{\big(1+|y-\overline{x}_{1}|\big)^{\frac{N+2} 2- \tau -m }}
 +\frac{1}{\big(1+|y-\overline{x}_{1}|\big)^{N+2}}\Big]
 \nonumber
 \\
 & \le  \frac{C}{{\bf r}^{m}}   \frac{1}{\big(1+|y-\overline{x}_{1}|\big)^{\frac{N+2} 2+\tau}},
\end{align*}
since $\frac{N+2} 2- \tau -m>0.$
Thus we have
\begin{align*}
 U_{\overline{x}_1,\Lambda}^{2^*-1} \,\Big[K\Big(\frac{|y|}{{\bf r}}\Big)-1\Big]
  \le  \frac{C}{{\bf r}^{ \min\{m, \frac{m+3} 2\} }}   \frac{1}{\big(1+|y-\overline{x}_{1}|\big)^{\frac{N+2} 2+\tau}}.
\end{align*}
  As a result,
\begin{equation} \label{S21}
 S_{21} \le  C\max\Big\{\Big(\frac k{\bf r} \Big)^{
\frac{N+2}{2}-\frac{N-2-m}{N-2}- \epsilon_1 },  \frac{1}{{\bf r}^{ \min\{m, \frac{m+3} 2\} }} \Big\}
\frac1{(1+|y-\overline{x}_1|)^{\frac{N+2}2+\tau}}.
 \end{equation}

Since  $ y \in \Omega_1^+ $, then for  $j= 2, \cdots, k $,  there holds
\begin{equation*}
|\overline{x}_1-\overline{x}_j|\le|y-\overline{x}_1|+|y-\overline{x}_j|\le 2|y-\overline{x}_j|.
\end{equation*}
Therefore, it is easy to derive that
\begin{align} \label{S22}
S_{22}&\le  C\frac{1}{(1+|y-\overline{x}_1|)^{\frac{N+2}2}}\sum_{j=2}^k \frac{1}{(1+|y-\overline {x}_j|)^{\frac{N+2}2}}
\nonumber \\[2mm]
& \le C  \frac{1}{(1+|y-\overline{x}_1|)^{\frac{N+2}2+\tau}} \sum_{j=2}^k \frac{1}{|\overline{x}_1-\overline{x}_j|^{\frac{N+2}2-\tau}}
\nonumber \\[2mm]
& \le  \frac{C}{(1+|y-\overline{x}_1|)^{\frac{N+2}2+\tau}} \Big(\frac k{\bf r} \Big)^{
\frac{N+2}{2}-\frac{N-2-m}{N-2}- \epsilon_1 }.
\end{align}
Combining \eqref{S21} with \eqref{S22}, we obtain
 \begin{equation*}
\|S_2\|_{**} \le C\max\Big\{\Big(\frac k{\bf r} \Big)^{
\frac{N+2}{2}-\frac{N-2-m}{N-2}- \epsilon_1 },  \frac{1}{{\bf r}^{ \min\{m, \frac{m+3} 2\} }} \Big\}.
 \end{equation*}

If $N=5$, we can check that $ \frac{1}{{\bf r}^{ m}}\,=\,\Big(\frac k{\bf r} \Big)^3 $. Thus, we can rewrite \eqref{S1} as
 \begin{align*}
 \|S_1\|_{**} \le  C\max\Big\{\Big(\frac k{\bf r} \Big)^{
\frac{N+2}{2}-\frac{N-2-m}{N-2}- \epsilon_1 },  \frac{1}{{\bf r}^{ \min\{m, \frac{m+3} 2\} }} \Big\}.
 \end{align*}
 Therefore, we showed  \eqref{estimateforlk}.
\end{proof}

\medskip
The  solvability theory  for the  projected problem \eqref{re1}  can be provided in the following:
\begin{proposition} \label{pro3}
Suppose that  $ K(|y|)$ satisfies  ${\bf H}$ and  $N\ge 5$, $(r,h,\Lambda) \in{{\mathscr S}_k}$. There exists an integer  $k_0$ large enough, such that for all  $k \ge k_0$ problem \eqref{re1} has a unique solution $\phi_k$ which satisfies
\begin{align} \label{estimateforphik}
\|\phi_k\|_{*} \le  C\max\Big\{ \frac1{k^{(\frac m{N-2-m})(
\frac{N+2}{2}-\frac{N-2-m}{N-2}- \epsilon_1 )}},  \frac1{k^{(\frac {N-2}{N-2-m}) \min\{m, \frac{m+3} 2\}}} \Big\},
\end{align}
 and
\begin{align}\label{estimateforc}
|c_{\ell}|\le   \frac C{(1+\delta_{\ell 2}{\bf r})} \max\Big\{ \frac1{k^{(\frac m{N-2-m})(
\frac{N+2}{2}-\frac{N-2-m}{N-2}- \epsilon_1 )}},  \frac1{k^{(\frac {N-2}{N-2-m}) \min\{m, \frac{m+3} 2\}}} \Big\},    \quad{for} ~ \ell=1,2,3.
\end{align}
\end{proposition}

\begin{proof}
 We first denote
\begin{align*}
\mathcal{B}:= \Bigg\{v:  v\in \mathbb{E} \quad\|v\|_* \le   C\max\Big\{ \frac1{k^{(\frac m{N-2-m})(
\frac{N+2}{2}-\frac{N-2-m}{N-2}- \epsilon_1 )}},  \frac1{k^{(\frac {N-2}{N-2-m}) \min\{m, \frac{m+3} 2\}}} \Big\} \Bigg\}.
 \end{align*}
 From Proposition \ref{p1}, we know that problem \eqref{re1} is equivalent to the following fixed point problem
\begin{align*}
 \phi= {\bf  L}_k \big({{\bf N}}(\phi)+{\bf l}_k\big)= : {\bf A}(\phi),
\end{align*}
 where  ${\bf  L}_k$ is the linear bounded operator defined in Proposition \ref{p1}.

From Lemma \ref{Lemma2.3} and Lemma \ref{Lemma2.4}, we know, for  $\phi \in \mathcal{B} $
\begin{align*}
\|{\bf A}(\phi)\|_*  & \le C  \Big(\|{{\bf N}}(\phi)\|_{**}+\|{\bf l}_k \|_{**}\Big)
\nonumber \\[2mm]
& \le\, O(\| \phi\|_*^{1+\sigma} )+ \max\Big\{ \frac1{k^{(\frac m{N-2-m})(
\frac{N+2}{2}-\frac{N-2-m}{N-2}- \epsilon_1 )}},  \frac1{k^{(\frac {N-2}{N-2-m}) \min\{m, \frac{m+3} 2\}}} \Big\}
\nonumber \\[2mm]
&  \le   \max\Big\{ \frac1{k^{(\frac m{N-2-m})(
\frac{N+2}{2}-\frac{N-2-m}{N-2}- \epsilon_1 )}},  \frac1{k^{(\frac {N-2}{N-2-m}) \min\{m, \frac{m+3} 2\}}} \Big\}.
\end{align*}
So the operator ${\bf A}$ maps from  $\mathcal{B}$ to $\mathcal{B}$. Furthermore, we can show that ${\bf A}$ is a contraction mapping.
In fact, for any  $\phi_1, \phi_2 \in \mathcal{B}$, we have
\begin{align*}
\|{\bf A}(\phi_1)-{\bf A}(\phi_2) \|_*  \le C\|{{\bf N}}(\phi_1)-{{\bf N}}(\phi_2)\|_{**}.
\end{align*}
 Since ${{\bf N}}(\phi)$ has a power-like behavior with power greater than one, then we can easily get
\begin{align*}
 \|{\bf A}(\phi_1)-{\bf A}(\phi_2) \|_* \le o(1)\| \phi_1-\phi_2\|_*.
\end{align*}
 A direct application of the contraction mapping principle yields that problem \eqref{re1} has a unique solution  $ \phi \in \mathcal{B}$.  The estimates for  $ c_{\ell}, \ell=1,2,3  $ can be got easily from \eqref{estimate cl}.
\end{proof}

%

  \section{Proof of Theorem \ref{main1}}  \label{sec4}

%

\begin{proposition} \label{pro2.4}
Let $ \phi_{r,h, \Lambda}$ be a function obtained in Proposition \ref{pro3} and
\begin{align*}
F(r,h,\Lambda):= I( W_{r, h, \Lambda}+\phi_{r,h, \Lambda}).
\end{align*}
If $(r,h, \Lambda)$ is a critical point of  $F(r,h,\Lambda)$, then
\begin{align*}
 u= W_{r,h,\Lambda}+\phi_{r,h,\Lambda}
\end{align*}
is a critical point of $I(u)$ in  $H^1(\mathbb{R}^N)$. \qed
\end{proposition}

\medskip
We will give the expression of  $F(r,h,\Lambda)$. We first note that we employ the notation  $\mathcal {C}(r, \Lambda)$ to denote functions which are independent of $h$ and  uniformly bounded.

\begin{proposition}\label{pr0position2,6}
Suppose that $K(|y|)$ satisfies ${\bf H}$ and $N\ge 5 $, $(r,h,\Lambda) \in{{\mathscr S}_k}$.  We have the following expansion as $k \to \infty$
\begin{align*}
F(r,h,\Lambda)
&\,=\, I(W_{r, h, \Lambda})+k O \Big(\frac{1}{k^{\big(\frac{m(N-2)}{N-2-m}+\frac{2(N-3)}{N-1}+\sigma\big)}}\Big)
\nonumber \\[2mm]
&\,=\, k  A_1 -\frac{k}{\Lambda^{N-2}} \Big[\,\frac{B_4 k^{N-2}}{(r \sqrt{1-h^2})^{N-2}}\,+\, \frac{B_5 k}{r^{N-2} h^{N-3} \sqrt {1-h^2}}\,\Big]
\nonumber \\[2mm]
& \quad+ k \Big[\frac{A_2}{\Lambda^{m}  k^{\frac{(N-2)m}{N-2-m}}}
+\frac{A_3}{\Lambda^{m-2}  k^{\frac{(N-2)m}{N-2-m}}}({\bf r}-r)^2\Big]
+k  \frac{\mathcal {C}(r, \Lambda)}{k^{\frac{m(N-2)}{N-2-m}}}({\bf r} -r)^{2+\sigma}
\nonumber\\[2mm]
& \quad
+k  \frac{\mathcal {C}(r, \Lambda)}{k^{\frac{m(N-2)}{N-2-m}+\sigma}}
+k O\Big(\frac{1}{k^{\big(\frac{m(N-2)}{N-2-m}+\frac{2(N-3)}{N-1}+\sigma\big)}}\Big),
\end{align*}
where  $A_1, A_2, A_3,  B_4, B_5$ are  positive constants.
\end{proposition}
\begin{proof}
The proof of Proposition \ref{pr0position2,6} is similar to that  of Proposition  $3.1 $  in \cite{WY-10-JFA}. We omit it here.
 \end{proof}

\medskip
Next, we will give the expansions of $\frac{\partial F(r,h, \Lambda)}{\partial \Lambda}$ and $ \frac{\partial F(r,h, \Lambda)}{\partial h} $.
\begin{proposition}
Suppose that $K(|y|)$ satisfies ${\bf H}$ and $N\ge 5 $, $(r,h,\Lambda) \in{{\mathscr S}_k}$.
We have the following expansion for $k \to \infty$
\begin{align}\label{thu25jun}
\frac{\partial F(r,h, \Lambda)}{\partial \Lambda}
&= \frac{k (N-2)}{\Lambda^{N-1}} \Big[\frac{B_4 k^{N-2}}{(r \sqrt{1-h^2})^{N-2}}\,+\, \frac{B_5 k}{r^{N-2} h^{N-3} \sqrt {1-h^2}} \Big]
\nonumber\\[2mm]
&\quad-k \Big[\, \frac{m A_2}{\Lambda^{m+1} k^{\frac{(N-2)m}{N-2-m}}}
+\frac{(m-2)A_3}{\Lambda^{m-1}  k^{\frac{(N-2)m}{N-2-m}}} ({\bf r}-r)^2\,\Big]
+kO \, \Big(\frac1{k^{\frac{(N-2)m}{N-2-m}+ \sigma}}\Big),
\end{align}
where  $A_2, A_3, B_4, B_5$ are  positive constants.
\end{proposition}
\begin{proof}
The proof of this proposition can be found in \cite{WY-10-JFA}. We omit it here.
\end{proof}

 \medskip
\begin{proposition}
Suppose that $K(|y|)$ satisfies ${\bf H}$ and $N\ge 5 $, $(r,h,\Lambda) \in{{\mathscr S}_k}$.
We have the following expansion
\begin{align} \label{frach}
\frac{\partial F(r,h, \Lambda)}{\partial h }
&\,=\,    -\frac{k}{\Lambda^{N-2}}\Big[\, (N-2) \frac{B_4 k^{N-2}}{ r^{N-2} (\sqrt{1-h^2})^{N}} h -(N-3) \frac{B_5 k}{r^{N-2} h^{N-2} \sqrt {1-h^2}} \,\Big]
\nonumber\\[2mm]
& \quad \quad   +kO\Big(\frac{1}{k^{\big(\frac{m(N-2)}{N-2-m}+\frac{(N-3)}{N-1}+\sigma\big)}}\Big),
\end{align}
where  $ B_4, B_5$ are  positive constants.
\end{proposition}
\begin{proof}
Notice that $F(r,h, \Lambda)\,=\, I( W_{r, h, \Lambda}+\phi_{r,h, \Lambda})$ , there holds
\begin{align}
&\frac{\partial F(r,h, \Lambda)}{\partial h }
 \nonumber\\[2mm]
&  =\left\langle  I'( W_{r, h, \Lambda}+\phi_{r,h, \Lambda}), \frac{\partial (W_{r, h, \Lambda}+\phi_{r,h, \Lambda}) } {\partial  h} \right\rangle
 \nonumber\\[2mm]
 &    = \left\langle  I'( W_{r, h, \Lambda}+\phi_{r,h, \Lambda}), \frac{\partial  W_{r, h, \Lambda}  } {\partial  h} \right\rangle
 +  \left\langle  I'( W_{r, h, \Lambda}+\phi_{r,h, \Lambda}), \frac{\partial  \phi_{r,h, \Lambda}  } {\partial  h}\right\rangle
  \nonumber\\[2mm]
 &    =   \left\langle  I'( W_{r, h, \Lambda}+\phi_{r,h, \Lambda}), \frac{\partial  W_{r, h, \Lambda}  } {\partial  h} \right\rangle  +  \left\langle \sum\limits_{j=1}^k\sum\limits_{\ell=1}^3
\Big(\, {c_\ell}U_{\overline{x}_j,\Lambda}^{2^*-2}\overline{\mathbb{Z}}_{\ell j}+{c_\ell}U_{\underline{x}_j,\Lambda}^{2^*-2}\underline{\mathbb{Z}}_{\ell j}\,\Big),    \frac{\partial  \phi_{r,h, \Lambda}  } {\partial  h}\right\rangle.
\end{align}
Since $ \int_{{\mathbb{R}}^N} U_{\overline{x}_j, \Lambda}^{2^*-2}  \overline{\mathbb{Z}}_{\ell j}  \phi_{r,h, \Lambda}=  \int_{{\mathbb{R}}^N} U_{\underline{x}_j, \Lambda}^{2^*-2}  \underline{\mathbb{Z}}_{\ell j}  \phi_{r,h, \Lambda}=0$, we can get easily
\begin{align*}
\left\langle U_{\overline{x}_j,\Lambda}^{2^*-2}\overline{\mathbb{Z}}_{\ell j}, \frac{\partial  \phi_{r,h, \Lambda}  } {\partial  h} \right\rangle
\,=\, - \left\langle \frac { \partial (U_{\overline{x}_j,\Lambda}^{2^*-2}\overline{\mathbb{Z}}_{\ell j})}{\partial h},\phi_{r,h, \Lambda} \right\rangle,
\end{align*}
\begin{align*}
\left\langle U_{\underline{x}_j,\Lambda}^{2^*-2}\overline{\mathbb{Z}}_{\ell j},    \frac{\partial  \phi_{r,h, \Lambda}  } {\partial  h}\ \right\rangle
\, =\,  -\left\langle \frac { \partial (U_{\overline{x}_j,\Lambda}^{2^*-2}\overline{\mathbb{Z}}_{\ell j})}{\partial h},      \phi_{r,h, \Lambda} \right\rangle.
\end{align*}
Then
\begin{align}
&\Bigg\langle  \sum\limits_{j=1}^k
 \Big(\, {c_\ell}U_{\overline{x}_j,\Lambda}^{2^*-2}\overline{\mathbb{Z}}_{\ell j}+{c_\ell}U_{\underline{x}_j,\Lambda}^{2^*-2}\underline{\mathbb{Z}}_{\ell j}\,\Big),    \frac{\partial  \phi_{r,h, \Lambda}  } {\partial  h}\Bigg\rangle
 \nonumber\\[2mm]
 &\leq C |c_\ell| \|  \phi_{r,h, \Lambda} \|_*   \sum\limits_{i=1}^k \int_{\R^N}     \frac { \partial (U_{\overline{x}_i,\Lambda}^{2^*-2}\overline{\mathbb{Z}}_{\ell i})}{\partial h}  \Biggl(
\sum_{j=1}^k \Big[\frac1{(1+|y-\overline{x}_j|)^{\frac{N-2}{2}+\tau
}}+\frac1{(1+|y-\underline{x}_j|)^{\frac{N-2}{2}+\tau
}}\Big] \Biggr)
 \nonumber\\[2mm]
 &\leq
 C |c_\ell| \|  \phi_{r,h, \Lambda} \|_*
  \nonumber\\[2mm]
 & \quad \times
  \sum\limits_{i=1}^k \int_{\R^N}    \frac {r(1+\delta_{\ell 2}{\bf r}) } {(1+ |y-\overline{x}_i| )^{N+3}  }   \Biggl(
\sum_{j=1}^k \Big[\frac1{(1+|y-\overline{x}_j|)^{\frac{N-2}{2}+\tau
}}+\frac1{(1+|y-\underline{x}_j|)^{\frac{N-2}{2}+\tau
}}\Big] \Biggr)
 \nonumber\\[2mm]
 &\leq  C  {\bf r}  \max\Big\{ \frac1{k^{(\frac m{N-2-m})(
N+2-2\frac{N-2-m}{N-2}- 2\epsilon_1 )}},  \frac1{k^{(\frac {N-2}{N-2-m}) \min\{2m, m+3\}}} \Big\},
 \end{align}
 where we used the estimates \eqref{estimateforphik}-\eqref{estimateforc} and  the inequalities
 \begin{align*}
 \Big|  \frac{ \partial  \big( U_{\overline{x}_i,\Lambda}^{2^*-2}\overline{\mathbb{Z}}_{ \ell i} \big)}{\partial h}  \Big|
  \le C   \frac {{\bf r}(1+\delta_{\ell 2} r) } {(1+ |y-\overline{x}_i| )^{N+3}  }\quad \text{for}~i=1, \cdots,k,  \ell=1,2,3.
\end{align*}

On the other hand, we have
\begin{align}\label{I'parW}
 & \Bigg\langle  I'( W_{r, h, \Lambda}+\phi_{r,h, \Lambda}), \frac{\partial  W_{r, h, \Lambda}  } {\partial  h} \Bigg\rangle
 \nonumber\\[2mm]
 & = \int_{\R^N}  \nabla ( W_{r, h, \Lambda}+\phi_{r,h, \Lambda}) \nabla W_{r, h, \Lambda} - \int_{\mathbb{R}^N} K\Big(\frac{|y|}{{\bf r}}\Big) ( W_{r, h, \Lambda}+\phi_{r,h, \Lambda})^{2^*-1} \frac{\partial  W_{r, h, \Lambda}  } {\partial  h}
  \nonumber\\[2mm]
 & =   \int_{\R^N}  \nabla   W_{r, h, \Lambda} \nabla \frac{\partial  W_{r, h, \Lambda}  } {\partial  h}  - \int_{\mathbb{R}^N} K\Big(\frac{|y|}{{\bf r}}\Big) ( W_{r, h, \Lambda}+\phi_{r,h, \Lambda})^{2^*-1} \frac{\partial  W_{r, h, \Lambda}  } {\partial  h}
  \nonumber\\[2mm]
 & =   \frac{\partial I( W_{r, h, \Lambda}) } {\partial h}  +   (2^*-1) \int_{\mathbb{R}^N} K\Big(\frac{|y|}{{\bf r}}\Big)  W_{r, h, \Lambda}^{2^*-2} \frac{\partial  W_{r, h, \Lambda}  } {\partial  h}  \phi_{r,h, \Lambda}   + O\Big( \int_{\mathbb{R}^N} \phi_{r,h, \Lambda}^2  \Big).
 \end{align}
 For the second term in \eqref{I'parW}, using the decay property of $K(|y|)$ and  orthogonality of $\phi_{r, h, \Lambda}$, we can show this term is small. In fact, we have

\begin{align*}
& \int_{\mathbb{R}^N} K\Big(\frac{|y|}{{\bf r}}\Big)  W_{r, h, \Lambda}^{2^*-2} \frac{\partial  W_{r, h, \Lambda}  } {\partial  h}  \phi_{r,h, \Lambda}
  \nonumber\\[2mm]
 & =    \int_{\mathbb{R}^N} K\Big(\frac{|y|}{{\bf r}}\Big)  \Big[  W_{r, h, \Lambda}^{2^*-2} \frac{\partial  W_{r, h, \Lambda}  } {\partial  h}   -  \sum_{i=1}^k \big(U_{\overline{x}_i,\Lambda}^{2^*-2} \overline{\mathbb{Z}}_{2i}  + U_{\underline{x}_i,\Lambda}^{2^*-2}  \underline{\mathbb{Z}}_{2i} \big) \Big] \phi_{r,h, \Lambda}
   \nonumber\\[2mm]
 & \quad + \sum_{i=1}^k    \int_{\mathbb{R}^N} \Big[K\Big(\frac{|y|}{{\bf r}}\Big)-1\Big]  \big(U_{\overline{x}_i,\Lambda}^{2^*-2} \overline{\mathbb{Z}}_{2i}  + U_{\underline{x}_i,\Lambda}^{2^*-2} \underline{\mathbb{Z}}_{2i} \big)  \phi_{r,h, \Lambda}
   \nonumber\\[2mm]
 & = 2k   \int_{\Omega^+_1}   K\Big(\frac{|y|}{{\bf r}}\Big)  \Big[  W_{r, h, \Lambda}^{2^*-2} \frac{\partial  W_{r, h, \Lambda}  } {\partial  h}   -  \sum_{i=1}^k \big(U_{\overline{x}_i,\Lambda}^{2^*-2} \overline{\mathbb{Z}}_{2i}  + U_{\underline{x}_i,\Lambda}^{2^*-2} \underline{\mathbb{Z}}_{2i} \big)  \Big] \phi_{r,h, \Lambda}
    \nonumber\\[2mm]
 & \quad + 2k \int_{\mathbb{R}^N} \Big[K\Big(\frac{|y|}{{\bf r}}\Big)-1\Big]   U_{\overline{x}_1,\Lambda}^{2^*-2}\overline{\mathbb{Z}}_{21}     \phi_{r,h, \Lambda}.
 \end{align*}
According to the expression of $W_{r, h, \Lambda}$, we can obtain that
 \begin{align*}
& \int_{\Omega^+_1}K\Big(\frac{|y|}{{\bf r}}\Big)\,\Big[  W_{r, h, \Lambda}^{2^*-2} \frac{\partial  W_{r, h, \Lambda}  } {\partial  h}
\,-\,\sum_{i=1}^k \big(U_{\overline{x}_i,\Lambda}^{2^*-2} \overline{\mathbb{Z}}_{2i}  + U_{\underline{x}_i,\Lambda}^{2^*-2}  \underline{\mathbb{Z}}_{2i} \Big] \phi_{r,h, \Lambda}
\nonumber\\[2mm]
&\le C \int_{\Omega^+_1}  \Big[ U_{\overline{x}_1,\Lambda}^{2^*-2}    \big( \sum_{j=2 }^k  \overline{\mathbb{Z}}_{2j}
\,+\,\sum_{j=1 }^k   \underline{\mathbb{Z}}_{2j}    \big)  + \big( \sum_{i=2}^k U_{\overline{x}_i,\Lambda}^{2^*-2} \overline{\mathbb{Z}}_{2i}
+\sum_{i=1}^k  U_{\underline{x}_i,\Lambda}^{2^*-2}\underline{\mathbb{Z}}_{2i} \big) \Big] \phi_{r,h, \Lambda}
\nonumber\\[2mm]
 &\le C\Big(\frac{k}{{\bf r}}\Big)^{\frac {N+2} 2 -\tau} \int_{\Omega^+_1}   \frac{\bf r}{\big(1+|y-\overline{x}_{1}|\big)^{\frac{N}{2}+2+\tau}} \phi_{r,h, \Lambda}
  \nonumber\\[2mm]
 &  \le C \Big(\frac{k}{{\bf r}}\Big)^{\frac {N+2} 2 -\tau}   \| \phi_{r,h, \Lambda}   \|_*  \int_{\Omega^+_1}   \frac{\bf r}{\big(1+|y-\overline{x}_{1}|\big)^{\frac{N}{2}+2+\tau}} \Bigl(
\sum_{j=1}^k \Big[\frac1{(1+|y-\overline{x}_j|)^{\frac{N-2}{2}+\tau
}}+\frac1{(1+|y-\underline{x}_j|)^{\frac{N-2}{2}+\tau
}}\Big] \Bigr)
\nonumber\\[2mm]
 &  \le C {\bf r} \Big(\frac{k}{{\bf r}}\Big)^{\frac {N+2} 2 -\tau} \| \phi_{r,h, \Lambda}   \|_* \le  C {\bf r} \Big(\frac k{\bf r} \Big)^{
\frac{N+2}{2}-\frac{N-2-m}{N-2}- \epsilon_1 }  \max\Big\{\Big(\frac k{\bf r} \Big)^{
\frac{N+2}{2}-\frac{N-2-m}{N-2}- \epsilon_1 },  \frac{1}{{\bf r}^{ \min\{m, \frac{m+3} 2\} }} \Big\}
\nonumber\\[2mm]
 &\leq  C  {\bf r}   \max\Big\{ \frac1{k^{(\frac m{N-2-m})(
N+2-2\frac{N-2-m}{N-2}- 2\epsilon_1 )}},  \frac1{k^{(\frac {N-2}{N-2-m}) \min\{2m, m+3\}}} \Big\}.
 \end{align*}
 And it's easy to show that
 \begin{align*}
& \int_{\mathbb{R}^N} \Big[K\Big(\frac{|y|}{{\bf r}}\Big)-1\Big]   U_{\overline{x}_1,\Lambda}^{2^*-2} \,\overline{\mathbb{Z}}_{21}\,\phi_{r,h, \Lambda}    \nonumber\\[2mm]
 &\leq  C  {\bf r}  \max\Big\{ \frac1{k^{(\frac m{N-2-m})(
N+2-2\frac{N-2-m}{N-2}- 2\epsilon_1 )}},  \frac1{k^{(\frac {N-2}{N-2-m}) \min\{2m, m+3\}}} \Big\}.
 \end{align*}
Combining all above,  we can get
\begin{align}  \label{ee}
\frac{\partial F(r,h, \Lambda)}{\partial h } & = \frac{\partial I(W_{r,h,\Lambda})}{\partial h}
 \nonumber\\[2mm]
 & \qquad
+  kO\Big( {\bf r}   \max\Big\{ \frac1{k^{(\frac m{N-2-m})(
N+2-2\frac{N-2-m}{N-2}- 2\epsilon_1 )}},  \frac1{k^{(\frac {N-2}{N-2-m}) \min\{2m, m+3\}}} \Big\} \Big).
\end{align}
Combing   \eqref{ee},  Proposition \ref{func3} and Lemma \ref{B6}, we can get \eqref{frach}
%
\end{proof}

\medskip
\begin{remark}\label{remark3.5}  {\em
The expansions of $\frac{\partial F(r,h, \Lambda)}{\partial h}$ and $\frac{\partial F(r,h, \Lambda)}{\partial \Lambda}$ would be applied in the proof of  Proposition \ref{5.73.3},  which is  essential  for proving  the existence  critical point of $F(r,h, \Lambda)$.
In order to get a proper  expansion  of $\frac{\partial F(r,h, \Lambda)}{\partial h}$, we need  accurate estimates for $\phi_{r,h, \Lambda} $.   }
\qed
\end{remark}


\medskip

{ \bf Rewritten the expansion of the energy functional.}

Let $\Lambda_0$ be
\begin{align} \label{lambda0}
	\Lambda_0=\Big[\,\frac{(N-2) B_4}{A_2 m}\,\Big]^{\frac 1 {N-2-m}}.
\end{align} Then it solves
\[\frac{B_4(N-2)}{\Lambda^{N-1}}- \frac{A_2 m}{\Lambda^{m+1}}=0.\]

\medskip
 Denote
\begin{align*}
 \mathcal {G}(h) &:=
 \frac{B_4 k^{N-2}}{(\sqrt{1-h^2})^{N-2}}\,+\, \frac{B_5 k}{h^{N-3} \sqrt {1-h^2}},
\end{align*}
 then
\begin{align*}
 \mathcal {G} '(h)
 &= (N-2) \frac{B_4 k^{N-2}h}{(\sqrt{1-h^2})^{N}}
\,-\,(N-3) \frac{B_5 k}{h^{N-2} \sqrt {1-h^2}}
\,+\,h \frac{B_5 k}{h^{N-4}({1-h^2})^{\frac 32}}
 \nonumber\\[2mm]
 &= (N-2)B_4 k^{N-2} h \big[\,1+O(h^2)\,\big]
 \nonumber\\[2mm]
 & \quad-(N-3) \frac{B_5 k}{h^{N-2}} \big[\,1+O(h^2)\,\big]+\frac{B_5 k}{h^{N-4}}  \big[\,1+O(h^2)\,\big]
 \nonumber\\[2mm]
 &=\Big[(N-2)  B_4 k^{N-2} h-(N-3) \frac{B_5 k}{h^{N-2}}\Big]+O\Big(\frac{k}{h^{N-4}}\Big),
\end{align*}
and
\begin{align}
 \mathcal {G} ''(h)
 \,=\,&(N-2)\frac{B_4 k^{N-2}}{(\sqrt{1-h^2})^{N}}
 +(N-2) N\frac{B_4 k^{N-2}h^2}{(\sqrt{1-h^2})^{N+2}}
 \nonumber\\[2mm]
 &+(N-3)(N-2)\frac{B_5 k}{h^{N-1} \sqrt {1-h^2}}
 -(N-3)\frac{B_5 k}{h^{N-3}({1-h^2})^{\frac 32}}
\nonumber\\[2mm]
&-(N-4)\frac{B_5 k}{h^{N-3}({1-h^2})^{\frac 32}}
+\frac{3B_5 k}{h^{N-5}({1-h^2})^{\frac 52}}\nonumber\\[2mm]
\,=\,&(N-2)  B_4 k^{N-2}
+(N-3)(N-2)\frac{B_5 k}{h^{N-1}}
+O\big(h^2  k^{N-2}\big)+O\Big(\frac{k}{h^{N-3}}\Big),
\end{align}
and
\begin{equation}
\mathcal {G} '''(h)\,=\,  O \Big(\frac{k}{h^{N}}\Big).
\end{equation}

\medskip
 Let  $ {\bf  h}$ be a solution of  $$\Big[(N-2)  B_4 k^{N-2} h-(N-3) \frac{B_5 k}{h^{N-2}}\Big]=0,$$
then
\begin{align} \label{B'}
{\bf h} = \frac{B'}{k^{\frac{N-3}{N-1}}}, \quad \text{with}   ~B'=  \Big[\frac{(N-3) B_5}{(N-2) B_4}\Big]^{\frac 1 {N-1}}.
\end{align}

\medskip
Define
\begin{align*}
{{\bf  S}_k}
=\Bigg\{&(r,h,\Lambda) \big|\,  r\in \Big[k^{\frac{N-2}{N-2-m}}-\frac{1}{{k^{\bar \theta}}}, k^{\frac{N-2}{N-2-m}}+\frac{1}{{k^{\bar \theta}}} \Big], \quad
\Lambda \in \Big[\Lambda_0-\frac{1}{{k^{\frac{3  \bar \theta}{2}}}},  \Lambda_0+\frac{1}{{k^{\frac{3  \bar \theta}{2}}}}\Big],
\nonumber\\[2mm]
& \qquad \qquad h \in \Big[\frac{B'}{k^{\frac{N-3}{N-1}}} \Big(1-\frac{1}{k^{\bar \theta}}\Big),\frac{B'}{k^{\frac{N-3}{N-1}}} \Big(1+\frac{1}{k^{\bar \theta}}\Big)\Big] \Bigg\},
\end{align*}
 for $\bar \theta$ is  a small constant such that $\bar \theta \le \frac \sigma {100}$.  In fact, $ {{\bf  S}_k}$ is a subset of ${{\mathscr  S}_k}$.   We will find a critical point of $F(r, h, \Lambda)$ in ${{\bf  S}_k}$.

\medskip
A direct Taylor  expansion gives that
 \begin{align}  \label{G1}
\mathcal {G}(h)= &\mathcal {G}({\bf h})+\mathcal {G} '({\bf h}) (h-{\bf h})+  \frac{1}{2}  \mathcal {G} '' ({\bf h}) (h-{\bf h})^2+O\Big(\mathcal {G} ''' \big({\bf h}+(1-\iota)  h \big)\Big)(h-{\bf h})^3,
\end{align}
where
\begin{align*}
\mathcal {G}({\bf h}) = B_4 k^{N-2}  \Big[1+\frac{N-2}{2} {\bf h}^2+O({\bf h}^4)\Big]
+\frac{B_5 k}{{\bf h}^{N-3}}  \Big[1+\frac{1} 2{\bf h}^2+O({\bf h}^4)\Big],
\end{align*}
\begin{align*}
\mathcal {G}'({\bf h})  = O\Big(\frac{k}{{\bf h}^{N-4}}\Big)
\end{align*}
and
\begin{align*}
\mathcal {G} '' ({\bf h}) =  \frac{(N-2)} 2 \Big[B_4 k^{N-2}+(N-3)  \frac{B_5 k}{{\bf h}^{N-1}} \Big] + O\big({\bf h}^2  k^{N-2}\big).
\end{align*}
   Since $\mathcal {G}({\bf h}), \mathcal {G} '' ({\bf h}) $ are independent of $h, r, \Lambda$, for simplicity, in the following, we  will denote
   \begin{align}  \label{G2}
\mathcal {G}({\bf h}) = B_4 k^{N-2}   + \frac{(N-2)B_4}{2}  k^{N-2}{\bf h}^2
+\frac{B_5 k}{{\bf h}^{N-3}},
\end{align}
\begin{align} \label{G3}
\mathcal {G} '' ({\bf h}) =  \frac{(N-2)} 2 \Big[B_4 k^{N-2}+(N-3)  \frac{B_5 k}{{\bf h}^{N-1}} \Big] .
\end{align}
Then combining \eqref{G1}, \eqref{G2},  \eqref{G3},  we can get
\begin{align*}
\mathcal {G}(h)
=&B_4 k^{N-2}  +  \frac{(N-2)B_4}{2}  k^{N-2}{\bf h}^2  +\frac{B_5 k}{{\bf h}^{N-3}}  + O\Big(\frac{k}{{\bf h}^{N-4}}\Big) (h-{\bf h})  \nonumber\\[2mm]
&
+\frac{(N-2)} 2 \Big[B_4 k^{N-2}+(N-3)  \frac{B_5 k}{{\bf h}^{N-1}} \Big](h-{\bf h})^2
+O\Big(\frac{k}{{\bf h}^{N}}\Big) (h-{\bf h})^3.
\end{align*}
%
Therefore, we get
\begin{align}
\mathcal {G}(h)= & B_4  k^{N-2}+\Big[\frac{(N-2) B_4 {B'}^2} 2+\frac{B_5}{{B'}^{N-3}}\Big] \frac{k^{N-2}}{k^{\frac{2(N-3)}{N-1}}}
\nonumber\\[2mm]
&+\frac{(N-2)} 2 \Big[B_4 {B'}^2+\frac{(N-3)  B_5}{{B'}^{N-3}} \Big]  \frac{k^{N-2}}{k^{\frac{2(N-3)}{N-1}}}(1-{\bf h}^{-1} h)^2
+O \Big(\frac{k^{N-2}}{k^{\frac{2(N-3)}{N-1}}}\Big) (1-{\bf h}^{-1} h)^3
\nonumber\\[2mm]
=& B_4  k^{N-2}+B_6  \frac{k^{N-2}}{k^{\frac{2(N-3)}{N-1}}}
+B_7 \frac{k^{N-2}}{k^{\frac{2(N-3)}{N-1}}}
(1-{\bf h}^{-1} h)^2 +O \Big(\frac{k^{N-2}}{k^{\frac{2(N-3)}{N-1}}}\Big) (1-{\bf h}^{-1} h)^3,
\end{align}
where
 \[ B_6= \frac{(N-2) B_4 {B'}^2} 2+\frac{B_5}{{B'}^{N-3}}  , \quad  B_7= \frac{(N-2)} 2 \Big[B_4 {B'}^2 +\frac{(N-3)  B_5}{{B'}^{N-3}} \Big].  \]

\medskip
 Since
 \[r\in \Big[k^{\frac{N-2}{N-2-m}}-\frac{1}{k^{\bar \theta}}, \quad k^{\frac{N-2}{N-2-m}}+\frac{1}{k^{\bar \theta}}\Big],\]
 then
\begin{align*}
r^{N-2}= k^{\frac{(N-2)^2}{N-2-m}} \Big(1+\frac{\mathcal {C}(r, \Lambda)}{{k^{\frac{(N-2)}{N-2-m}+\bar \theta}}}\Big).
\end{align*}
We now rewrite
\begin{align*}
 & \frac{B_4 k^{N-2}}{(r \sqrt{1-h^2})^{N-2}}\,+\,\frac{B_5 k}{r^{N-2} h^{N-3}\sqrt {1-h^2}}
\nonumber\\[2mm]
&=\frac{B_4}{k^{\frac{(N-2)m}{N-2-m}}}+\frac{B_6}{k^{\frac{(N-2)m}{N-2-m}+\frac{2(N-3)}{N-1}}}+\frac{\mathcal {C}(r, \Lambda)}{k^{\frac{(N-2)m}{N-2-m}+\sigma}}
\nonumber\\[2mm]
&\quad+\frac{B_7}{k^{\frac{(N-2)m}{N-2-m}+\frac{2(N-3)}{N-1}}} (1-{\bf h}^{-1} h)^2+O \Big(\frac{1}{k^{\frac{(N-2)m}{N-2-m}+\frac{2(N-3)}{N-1}}}\Big) (1-{\bf h}^{-1} h)^3.
\end{align*}
 Then we can express $ F(r,h, \Lambda)$ as
\begin{align}\label{Frh}
F(r,h,\Lambda)\,=\, &k  A_1 -k \Big[\frac{B_4}{\Lambda^{N-2} k^{\frac{(N-2)m}{N-2-m}}}+\frac{B_6}{\Lambda^{N-2} k^{\frac{(N-2)m}{N-2-m}+\frac{2(N-3)}{N-1}}}
\nonumber\\[2mm]
&+\frac{B_7}{\Lambda^{N-2}  k^{\frac{(N-2)m}{N-2-m}+\frac{2(N-3)}{N-1}}} (1-{\bf h}^{-1} h)^2\Big]
\nonumber\\[2mm]
&+k \Big[\frac{A_2}{\Lambda^{m}  k^{\frac{(N-2)m}{N-2-m}}}+\frac{A_3}{\Lambda^{m-2}  k^{\frac{(N-2)m}{N-2-m}}}({\bf r}-r)^2\Big] +k\frac{\mathcal {C}(r, \Lambda)}{k^{\frac{(N-2)m}{N-2-m}}} ({\bf r}-r)^{2+\sigma}  \nonumber
\\[2mm]
&+k O \Big(\frac{1}{k^{\frac{(N-2)m}{N-2-m}+\frac{2(N-3)}{N-1}}}\Big) (1-{\bf h}^{-1} h)^3+ k O\Big(\frac{1}{k^{\big(\frac{m(N-2)}{N-2-m}+\frac{2(N-3)}{N-1}+\sigma\big)}}\Big).
\end{align}

And similarly, we have
\begin{align*}
\frac{\partial F(r,h,\Lambda)}{\partial \Lambda}
\,=\,&k  \Big[\frac{ (N-2) B_4}{\Lambda^{N-1} k^{\frac{(N-2)m}{N-2-m}}} - \frac{m A_2}{\Lambda^{m+1} k^{\frac{(N-2)m}{N-2-m}}} \Big]
\nonumber\\[2mm]
&+\frac{(m-2)A_3}{\Lambda^{m-1}  k^{\frac{(N-2)m}{N-2-m}}} ({\bf r}-r)^2
+kO \, \Big(\frac1{k^{\frac{(N-2)m}{N-2-m}}} ({\bf r}-r)^{2+\sigma}\Big);
\end{align*}
and  from \eqref{frach}, by using some calculations, we have
\begin{align} \label{frach1}
\frac{\partial F(r,h, \Lambda)}{\partial h }
&\,=\,     \frac{k}{\Lambda^{N-2}}\Big[\,  \frac{2B_7}{\Lambda^{N-2}  k^{\frac{(N-2)m}{N-2-m}+\frac{(N-3)}{N-1}}} (1-{\bf h}^{-1} h)  \,\Big]     \nonumber\\[2mm]
& \quad
  +kO\Big(\frac{1}{k^{\big(\frac{m(N-2)}{N-2-m}+\frac{(N-3)}{N-1} \big)}}\Big) (1-{\bf h}^{-1} h)^2 +kO\Big(\frac{1}{k^{\big(\frac{m(N-2)}{N-2-m}+\frac{(N-3)}{N-1}+\sigma\big)}}\Big).
\end{align}

{\red
}

 Now define
 \begin{align}\label{barF}
 \bar  F(r,h,\Lambda)=-F(r,h,\Lambda),
 \end{align}
 and
 \[{\bf t}_2= k (-A_1+\eta_1),
 \quad {\bf t}_1
= k  \Big(-A_1 -\big( \frac{A_2}{\Lambda_0^{m}} - \frac{B_4}{\Lambda_0^{N-2}}  \big) \frac{1}{k^{\frac{(N-2)m}{N-2-m}}} -\frac{1}{k^{\frac{(N-2) m}{N-2-m}+\frac{5 \bar \theta} 2}}\Big),  \]
 where $\eta_1>0 $ small. We also define the energy level set
 \[  \bar F^{\bf t}= \Big\{(r,h,\Lambda) \big|~(r,h,\Lambda) \in {{\bf  S}_k}, ~ \bar F (r,h,\Lambda) \le{\bf t} \Big\}.  \]

We consider the following  gradient flow system
\begin{equation*}
\begin{cases}
\frac{{\mathrm d} r}{{\mathrm d} t}=-{\bar F}_r , &t>0;\\[2mm]
\frac{{\mathrm d} h}{{\mathrm d} t}=-{\bar F}_h , &t>0;\\[2mm]
\frac{{\mathrm d} \Lambda}{{\mathrm d} t}=-{\bar F}_\Lambda, &t>0;\\[2mm]
(r,h,\Lambda) \big|_{t=0}  \in \bar F^{{\bf t}_2}.
\end{cases}
\end{equation*}
The next proposition  would play an important role in the proof of Theorem \ref{main1}.

\begin{proposition} \label{5.73.3}
The flow would not leave  $ {{\bf  S}_k} $ before it reaches  $ \bar F^{{\bf t}_1}. $
\end{proposition}
\begin{proof}
There are   three  positions that the flow tends to leave  $ {{\bf  S}_k} $:
\\[2mm]
 \item[ {\bf position  1.} ]
 $|r-{\bf r}|= \frac{1}{k^{\bar \theta}}  $ and  $|1-{\bf h}^{-1} h| \le \frac{1}{k^{\bar \theta}}, \quad|\Lambda-\Lambda_0|\le \frac{1}{k^{\frac{3  \bar \theta}{2}}}  $;
 \\[2mm]
 ~\item[ {\bf position  2.} ] $|1-{\bf h}^{-1} h|=  \frac{1}{k^{\bar \theta}}  $  when  $|r-{\bf r}| \le \frac{1}{k^{\bar \theta}}, \quad|\Lambda-\Lambda_0|\le \frac{1}{k^{\frac{3  \bar \theta}{2}}} $;
 \\[2mm]
 ~ \item[ {\bf position  3.} ]  $| \Lambda-\Lambda_0|=  \frac{1}{k^{\frac{3  \bar \theta}{2}}} $ when  $|r-{\bf r}| \le \frac{1}{k^{\bar \theta}}, \quad|1-{\bf h}^{-1} h| \le \frac{1}{k^{\bar \theta}}. $
 \\[2mm]

\noindent $\spadesuit $
 We now consider {\bf position  1}.  Since  $|\Lambda-\Lambda_0|\le \frac{1}{k^{\frac{3  \bar \theta}{2}}}  $, it is easy to derive that
\begin{align} \label{decomlambda} \begin{split}
\Big(\frac{B_4}{\Lambda^{N-2}} - \frac{A_2}{\Lambda^{m}}\Big)
 &= \Big(\frac{B_4}{\Lambda_0^{N-2}} - \frac{A_2}{\Lambda_0^{m}}\Big)+O(|\Lambda-\Lambda_0|^2)\\[2mm]
 &=\Big(\frac{B_4}{\Lambda_0^{N-2}} - \frac{A_2}{\Lambda_0^{m}}\Big)+O\Big(\frac{1}{k^{3 \bar \theta}}\Big).
\end{split}
\end{align}
Combining \eqref{Frh}, \eqref{barF}, \eqref{decomlambda},  we can obtain that, if  $(r,h,\Lambda) $ lies in {\bf position  1},
\begin{align*}
  \bar F(r,h,\Lambda)  &=-k  A_1+k \Big[\frac{B_4}{\Lambda_0^{N-2} k^{\frac{(N-2)m}{N-2-m}}} -
 \frac{A_2}{\Lambda_0^{m}  k^{\frac{(N-2)m}{N-2-m}}} \Big] \nonumber
  \\[2mm]
  & \quad
- k  \frac{A_3}{\Lambda_0^{m-2}  k^{\frac{(N-2) m}{N-2-m}+2 \bar\theta}}
+O\Big(\frac{1}{k^{\frac{(N-2) m}{N-2-m}+\frac{5 \bar \theta} 2}}\Big) <{\bf t}_1.
\end{align*}

\noindent $\spadesuit$ On the other hand, we claim that it's impossible for the flow  $\big(r(t), h(t), \Lambda(t)\big)$ leaves  ${{\bf S}_k}$ when it lies in {\bf position  2}.
If $1-{\bf h}^{-1} h=  \frac{1}{k^{\bar \theta}} $, then from \eqref{frach1} and \eqref{barF}, we have
\begin{align}
\frac{\partial \bar F(r,h,\Lambda)}{\partial h} = -\frac{k}{\Lambda^{N-2}}\Big[\,  \frac{2B_7}{\Lambda^{N-2}  k^{\frac{(N-2)m}{N-2-m}+\frac{(N-3)}{N-1} +\bar\theta}} \,\Big]  + O\Big(\frac{1}{k^{\frac{(N-2) m}{N-2-m}+ \frac{N-3} {N-1} +2 \bar \theta}}\,\Big) <0.
\end{align}
On the other hand, if  $1-{\bf h}^{-1} h=  -\frac{1}{k^{\bar \theta}} $
\begin{align}
\frac{\partial \bar F(r,h,\Lambda)}{\partial h} = \frac{k}{\Lambda^{N-2}}\Big[\,  \frac{2B_7}{\Lambda^{N-2}  k^{\frac{(N-2)m}{N-2-m}+\frac{(N-3)}{N-1} +\bar\theta}} \,\Big]  + O\Big(\frac{1}{k^{\frac{(N-2) m}{N-2-m}+ \frac{N-3} {N-1} +2 \bar \theta}}\,\Big) >0.
\end{align}

So it's  impossible for  the flow leaves  ${{\bf S}_k}$ when it lies in {\bf position  2}.

\medskip
\noindent $\spadesuit $ Finally, we consider {\bf position  3}.
 If  $\Lambda=\Lambda_0+\frac{1}{k^{\frac{3  \bar \theta}{2}}} $, from
 \eqref{thu25jun} and \eqref{barF},  there exists a constant  $ C_1 $ such that
\[
\frac{\partial \bar F(r,h,\Lambda)}{\partial \Lambda}
= k\Big[\,C_1 \frac{1}{k^{\frac{(N-2) m}{N-2-m}+\frac 32 \bar \theta}}
+O\Big(\frac{1}{k^{\frac{(N-2) m}{N-2-m}+2 \bar \theta}}\Big)\,\Big]
>0.
\]
On the other hand, if  $\Lambda=\Lambda_0-\frac{1}{k^{\frac{3  \bar \theta}{2}}}$, there exists a constant  $C_2$ such that
\[
\frac{\partial \bar F(r,h,\Lambda)}{\partial \Lambda}
= k\Big[\,-C_2 \frac{1}{k^{\frac{(N-2) m}{N-2-m}+\frac 32 \bar \theta}}
+O\Big(\frac{1}{k^{\frac{(N-2) m}{N-2-m}+2 \bar \theta}}\,\Big)\Big]<0.
\]
Hence the flow  $ \big(r(t), h(t), \Lambda(t)\big)$ does not leave  ${{\bf S}_k}$ when  $| \Lambda-\Lambda_0|=  \frac{1}{k^{\frac{3  \bar \theta}{2}}}$.

Combining above results, we conclude that the flow would not leave  ${{\bf S}_k}$ before it reach  $ \bar{F}^{{\bf t}_1}$.
 \end{proof}

 \vspace{3mm}
 Now we give the proof of Theorem \ref{main1}. \\[2mm]
{\textit{Proof of Theorem \ref{main1}}}:
According to  Proposition \ref{pro2.4}, in order to show Theorem \ref{main1},  we only need to  show that function  $\bar F(r,h, \Lambda)$, and thus $F(r,h, \Lambda)$,  has a critical point in  ${{\bf  S}_k}$.

Define
\[
\begin{split}
\Gamma=\Bigl\{\gamma :\quad &\gamma(r,h,\Lambda)= \big(\gamma _1(r,h,\Lambda),\gamma _2(r,h,\Lambda),\gamma _3(r,h,\Lambda)\big)\in{{\bf  S}_k},(r,h,\Lambda)\in{{\bf S}_k}; \\[2mm]
&\gamma(r,h,\Lambda)=(r,h,\Lambda), \;\text{if} ~|r-{\bf r}|=  \frac{1}{k^{\bar \theta}}  \Bigr\}.
\end{split}
\]
Let
\[
{\bf c}=\inf_{\gamma  \in\Gamma}\max_{(r,h,\Lambda)\in{{\bf  S}_k}} \bar F \big(\gamma(r,h,\Lambda)\big).
\]

We claim that  $ {\bf c}  $ is a critical value of  $ \bar F(r,h, \Lambda)  $ and can be achieved by some  $(r,h, \Lambda) \in{{\bf S}_k} $. By the minimax theory, we need to show that
\medskip
 \begin{itemize}

\item[(i)]  $ {\bf t}_1< {\bf c} < {\bf t}_2 $;
\medskip

\item[(ii)] \label{itemii}  $\sup_{|r-{\bf r}|=  \frac{1}{k^{\bar \theta}}}
 {\bar F}\big(\gamma(r,h,\Lambda)\big)<{\bf t}_1,\;\forall\; \gamma \in \Gamma. $

\end{itemize}
Using  the results in Proposition \ref{5.73.3} we can prove (i) and (ii) easily.

  Finally, for every  $ k $ large enough, we get the critical point  $(r_k, h_k, \Lambda_k)  $ of  $ F(r,h,\Lambda) $.  %
 \qed

\appendix

%

\section{expansions for the energy functional}\label{appendixA}
 This section is devoted to the computation of the expansion for the energy functional  $I(W_{r,h,\Lambda})  $.  We first give the following Lemma.
 \begin{lemma}  \label{express7}
  $N\ge 5 $ and  $(r,h,\Lambda) \in{{\mathscr  S}_k}$.
We have the following expansions for $k \to \infty$:
\begin{align} \label{express3}
\sum_{i=2}^k  \frac{1}{|\overline{x}_1-\overline{x}_i|^{N-2}} =  \frac{k^{N-2}}{\big(r \sqrt{1-h^2}\big)^{N-2}}  \big(B_1+\sigma_1(k)\big),
\end{align}
 \begin{equation}
 \begin{split} \label{express4}
& \sum_{i=1}^k  \frac{1}{|\overline{x}_1-\underline{x}_i|^{N-2}} = \frac{B_2 k}{r^{N-2} h^{N-3} \sqrt {1-h^2}} \, \big(1+\sigma_2(k)\big)+\frac{\sigma_1(k) k^{N-2}}{\big(r \sqrt{1-h^2}\big)^{N-2}},
\end{split}
\end{equation}
where
\begin{align} \label{B1sigmak}
B_1=  \frac 2 {(2\pi)^{N-2}}  \sum_{i=1}^{\infty} \frac 1 {i^{N-2}},
\quad B_2= \frac 1 {2^{N-3} \pi}\int_0^{+\infty}\,\frac{1}{\big(s^2+1\big)^{\frac{N-2}{2}}}\,{\mathrm d}s,
\end{align}
 and
\begin{align} \label{sigmak}
\sigma_1(k)=
\begin{cases}
O\big(\frac{1}{k^{2}}\big),  \quad  N\ge 6,
\\[2mm]
O\big(\frac{\ln k}{k^2}\big), \quad N=5,
\end{cases}  \quad \sigma_2(k)= O\big((hk)^{-1}\big).
\end{align}
\end{lemma}

\begin{proof}
In fact, for  $ \frac{1}{2} < c_3 \leq c_4\leq 1 $,  we have
\begin{align*}
 c_3 \frac{i \pi}{k}
 \leq  \sin {\frac{i\pi}{k}} \leq c_4 \frac{i \pi}{k}, \quad \text{for} ~i \in  \big\{1, \cdots,  \frac k2 \big\}.
\end{align*}
 Without loss of generality, we can assume  $ k $ is even. It is easy to derive that
\begin{align*}
& \sum_{i=2}^k  \frac{1}{|\overline{x}_1-\overline{x}_i|^{N-2}}
= \sum_{i=1}^k \Big(\frac{1}{2r \sqrt{1-h^2} \sin{\frac{i\pi}{k}}}\Big)^{N-2} \nonumber
\\[2mm]
&=  \sum_{i=1}^{\frac k2} \,\Big(\frac{1}{2r \sqrt{1-h^2} \sin{\frac{i \pi}{k}}}\Big)^{N-2}+\sum_{i= \frac k2+1}^{k} \,\Big(\frac{1}{2r \sqrt{1-h^2} \sin{\frac{i \pi}{k}}}\Big)^{N-2}.
\end{align*}
Direct computations show that
 \begin{align} \label{sepfri11}
 &  \sum_{i=1}^{\frac k2} \,\Big(\frac{1}{2r \sqrt{1-h^2} \sin{\frac{i \pi}{k}}}\Big)^{N-2} \nonumber
\\[2mm]
 =
 & \sum_{i=1}^{{[\frac k 6]}} \Big(\frac{1}{2r \sqrt{1-h^2} \sin{\frac{i\pi}{k}}}\Big)^{N-2}
+\sum_{i= {{[\frac k 6]+1}}}^{\frac k 2} \Big(\frac{1}{2r \sqrt{1-h^2} \sin{\frac{i \pi}{k}}}\Big)^{N-2}  \nonumber
 \\[2mm]
 &= \sum_{i=1}^{{[\frac k 6]}}\Big(\frac{1}{2r \sqrt{1-h^2}{\frac{i \pi}{k}}}\Big)^{N-2} \Big(1+O\Big({\frac{i^2}{k^2}}\Big)\Big)
+O \Big(\frac{k}{\big(2r \sqrt{1-h^2}\big)^{N-2}}\Big) \nonumber
 \\[2mm]
 &= \Big(\frac{k}{r \sqrt{1-h^2}}\Big)^{N-2} \Big(D_{1}+\sigma_1(k)\Big),
\end{align}
where  $  D_{1}= \frac 1 {{2\pi}^{N-2}}  \sum_{i=1}^{\infty} \frac 1 {i^{N-2}} $
and  $\sigma_1(k)  $ is defined in \eqref{sigmak}.  Using symmetry of function  $\sin x $, we can easily show
\begin{align*}
\sum_{i= \frac k2+1}^{k} \,\Big(\frac{1}{2r \sqrt{1-h^2} \sin{\frac{i \pi}{k}}}\Big)^{N-2}
= \Big(\frac{k}{r \sqrt{1-h^2}}\Big)^{N-2} \Big(D_{1}+\sigma_1(k)\Big).
\end{align*}
Thus we proved \eqref{express3}.

Similarly, we can obtain
\begin{equation*}
\begin{split}
& \sum_{i=1}^k  \frac{1}{|\overline{x}_1-\underline{x}_i|^{N-2}}
= \sum_{i=1}^k \frac{1}{\Big(2r \big[(1-h^2) \sin^2 {\frac{(i-1) \pi}{k}}+h^2\big]^{\frac{1}{2}}\Big)^{N-2}}
 \\[2mm]
 &= \frac{2}{(2r h)^{N-2}} \sum_{i=1}^{\frac k2} \frac{1}{\Big(\frac{(1-h^2)}{h^2}\frac{(i-1)^2 \pi^2}{k^2}+1\Big)^{\frac{N-2}{2}}}+\sigma_1(k)  O \Big(\Big(\frac{k}{r \sqrt{1-h^2}}\Big)^{N-2}\Big).
\end{split}
\end{equation*}

%

Consider  $  O \big((h k)^{-1}\big)= o(1)  ~\text{as}~  k \rightarrow \infty  $.
Since\begin{align*}
 & \sum_{j=1}^{\frac k2}  \frac{1}{\Big(\frac{(1-h^2)}{h^2} \frac{(j-1)^2 \pi^2}{k^2}+1\Big)^{\frac{N-2}{2}}}
 \ge \int_0^{\frac k2}\,\frac{1}{\Big(\frac{(1-h^2)}{h^2}  \frac{x^2 \pi^2}{k^2}+1\Big)^{\frac{N-2}{2}}}\,{\mathrm d}x \, \nonumber
 \\[2mm] &
 \ge \int_0^{2}\,\frac{1}{\Big(\frac{(1-h^2)}{h^2}  \frac{x^2 \pi^2}{k^2}+1\Big)^{\frac{N-2}{2}}}\,{\mathrm d}x
 \,+\sum_{j=4}^{{\frac k2+1}}  \frac{1}{\Big(\frac{(1-h^2)}{h^2} \frac{(j-1)^2 \pi^2}{k^2}+1\Big)^{\frac{N-2}{2}}},
\end{align*}
  then we have
\begin{align*}
 & \sum_{j=1}^{\frac k2}\frac{1}{\Big(\frac{(1-h^2)}{h^2} \frac{(j-1)^2 \pi^2}{k^2}+1\Big)^{\frac{N-2}{2}}} \nonumber
 \\[2mm]
 &=  \int_0^{\frac k2} \, \frac{1}{\Big(\frac{(1-h^2)}{h^2} \frac{x^2 \pi^2}{k^2}+1\Big)^{\frac{N-2}{2}}}\,{\mathrm d}x+1+o(1)
 \nonumber\\[2mm]
 &=  \frac{hk}{\sqrt{1-h^2} \pi}\, \int_0^{\frac{(1-h^2)}{4 h^2} \pi^2}\,\frac{1}{\big(s^2+1\big)^{\frac{N-2}{2}}}\,{\mathrm d}s  +1+o(1)
 \nonumber\\[2mm]
 &=  \frac{hk}{\sqrt{1-h^2} \pi}\, \int_0^{+\infty}\,\frac{1}{\big(s^2+1\big)^{\frac{N-2}{2}}}\,{\mathrm d}s  \Big(1+O\big( (kh)^{-1} \big)\Big).
\end{align*}

%
Combining above calculations, we can obtain that
\begin{align*}
& \sum_{i=1}^k  \frac{1}{|\overline{x}_1-\underline{x}_i|^{N-2}} = \frac{1}{(r h)^{N-2}}  \frac{B_2 hk}{\sqrt{1-h^2}} \Big(1+\sigma_2 (k) \Big)
+O \Big(\frac{\sigma_1(k) k^{N-2}}{\big(r \sqrt{1-h^2}\big)^{N-2}}\Big),
\end{align*} where  $ B_2 $ and  $ \sigma_2 $  are defined in \eqref{B1sigmak}, \eqref{sigmak}.

%
%
 \end{proof}

 \begin{lemma}
 We have the expansion, for $k \to \infty$
\begin{align*}
\int_{\R^N}  U_{\overline{x}_1, \Lambda}^{2^*-1}  U_{\overline{x}_i, \Lambda}
= \frac{B_0}{\Lambda^{N-2}|\overline{x}_1-\overline{x}_i|^{N-2}}+O\Big(\frac 1 {|\overline{x}_1-\overline{x}_i|^{N-\epsilon_0}}\Big),  
\end{align*}
 and
\begin{align*}
\int_{\R^N}  U_{\overline{x}_1, \Lambda}^{2^*-1}  U_{\underline{x}_i, \Lambda}
=\frac{B_0}{\Lambda^{N-2}|\overline{x}_1-\underline{x}_i|^{N-2}}+O\Big(\frac 1 {|\overline{x}_1-\underline{x}_i|^{N-\epsilon_0}}\Big),  \quad 
\end{align*}
 where  $B_0= \int_{\mathbb{R}^N}  \frac{1}{(1+z^{2})^{\frac{N+2} 2}} $ and  $ \epsilon_0 $ is constant small enough.
 \end{lemma}
 \begin{proof}
 Let  $ \overline{d}_j =|\overline{x}_1-\overline{x}_j|, ~ \underline{d}_j =|\overline{x}_1-\underline{x}_j| $ for  $j=1, \cdots, k $. We consider
\begin{align} \label{A10}
 \int_{\mathbb{R}^N}  U_{\overline{x}_1, \Lambda}^{2^*-1} U_{\overline{x}_i, \Lambda}
 &= \int_{\mathbb{R}^N}  \frac{\Lambda^{\frac{N+2} 2}}{(1+\Lambda^2|y-\overline{x}_1|^{2})^{\frac{N+2} 2}}  \frac{\Lambda^{\frac{N-2}{2}}}{(1+\Lambda^2|y-\overline{x}_i|^{2})^{\frac{N-2}{2}}}  \nonumber
 \\[2mm]
  &= \Bigg\{\int_{B_{\frac{\overline{d}_ i} 4}(\overline{x}_1)}+\int_{\mathbb{R}^N \setminus {B_{\frac{\overline  d_i} 4}(\overline{x}_1)}}\Bigg\} \frac{\Lambda^{\frac{N+2} 2}}{(1+\Lambda^2|y-\overline{x}_1|^{2})^{\frac{N+2} 2}}  \frac{\Lambda^{\frac{N-2}{2}}}{(1+\Lambda^2|y-\overline{x}_i|^{2})^{\frac{N-2}{2}}}.
\end{align}

First, we have
\begin{align} \label{cal1}
 & \int_{B_{\frac{\overline{d}_i} 4}(\overline{x}_1)}  \frac{\Lambda^{\frac{N+2} 2}}{(1+\Lambda^2|y-\overline{x}_1|^{2})^{\frac{N+2} 2}}  \frac{\Lambda^{\frac{N-2}{2}}}{(1+\Lambda^2|y-\overline{x}_i|^{2})^{\frac{N-2}{2}}}
 \nonumber\\[2mm]
 & = \int_{B_{\frac{\Lambda \overline{d}_i} 4}(0)}  \frac{1}{(1+z^{2})^{\frac{N+2} 2}}  \frac{1}{(1+z^{2}+2\Lambda \langle z , \overline{x}_1-\overline{x}_i \rangle +\Lambda^2|\overline{x}_1-\overline{x}_i|^2)^{\frac{N-2}{2}}}
 \\[2mm]
 &= \frac 1 {\Lambda^{N-2}|\overline{x}_1-\overline{x}_i|^{N-2}} \int_{B_{\frac{\Lambda \overline{d}_i} 4}(0)}  \frac{1}{(1+z^{2})^{\frac{N+2} 2}} \Bigg(1-\frac{N-2}{2}  \frac{1+z^{2}+2 \Lambda\langle z , \overline{x}_1-\overline{x}_i \rangle}{\Lambda^{2}|\overline{x}_1-\overline{x}_i|^{2}}  \nonumber
 \\[2mm]
 &\qquad \qquad+O \Big(\Big(\frac{1+z^{2}+2  \Lambda\langle z , \overline{x}_1-\overline{x}_i \rangle}{\Lambda^{2}|\overline{x}_1-\overline{x}_i|^{2}}
\Big)^2\Big)
  \Bigg). \nonumber
\end{align}
It is easy to check that
\begin{align} \label{cal2}
\frac 1 {\Lambda^{N-2}|\overline{x}_1-\overline{x}_i|^{N-2}}  O \Big(\int_{B_{\frac{\Lambda \overline{d}_i} 4}(0)}  \frac{1}{(1+z^{2})^{\frac{N+2} 2}} \Big(\frac{1+z^{2}+2\Lambda\langle z , \overline{x}_1-\overline{x}_i \rangle}{\Lambda^{2}|\overline{x}_1-\overline{x}_i|^{2}}
\Big)^2 \Big)= O \Big(\frac 1 {|\overline{x}_1-\overline{x}_i|^{N}}\Big),
\end{align}
and
\begin{align}  \label{cal3}
 &\frac 1 {\Lambda^{N}|\overline{x}_1-\overline{x}_i|^{N}} \int_{B_{\frac{\Lambda \overline{d}_i} 4}(0)}  \frac{1}{(1+z^{2})^{\frac{N+2} 2}} \Big(1+z^{2}+2 \Lambda\langle z , \overline{x}_1-\overline{x}_i \rangle\Big) = O\Big(\frac 1 {|\overline{x}_1-\overline{x}_i|^{N-\epsilon_0}}\Big).
 \end{align}
Standard calculation implies that
\begin{align}  \label{cal4}
 \frac 1 {\Lambda^{N-2}|\overline{x}_1-\overline{x}_i|^{N-2}} \int_{B_{\frac{\Lambda \overline{d}_i} 4}(0)}  \frac{1}{(1+z^{2})^{\frac{N+2} 2}}
 =\frac{B_0}{\Lambda^{N-2}|\overline{x}_1-\overline{x}_i|^{N-2}}
 +O\Big(\frac 1 {|\overline{x}_1-\overline{x}_i|^{N}}\Big),
\end{align}
 where  $B_0= \int_{\mathbb{R}^N}  \frac{1}{(1+z^{2})^{\frac{N+2} 2}}  $.

From \eqref{cal1}-\eqref{cal4},  we get
\begin{align}  \label{A.15}
 & \int_{B_{\frac{\overline{d}_i} 4}(\overline{x}_1)} \frac{\Lambda^{\frac{N+2} 2}}{(1+\Lambda^2|y-\overline{x}_1|^{2})^{\frac{N+2} 2}}  \frac{\Lambda^{\frac{N-2}{2}}}{(1+\Lambda^2|y-\overline{x}_i|^{2})^{\frac{N-2}{2}}}
\nonumber
 \\[2mm]
 &= \frac{B_0}{\Lambda^{N-2}|\overline{x}_1-\overline{x}_i|^{N-2}}+O\Big(\frac 1 {|\overline{x}_1-\overline{x}_i|^{N-\epsilon_0}}\Big).
 \end{align}

When  $ y \in \mathbb{R}^N \setminus {B_{\frac{\overline  d_i} 4}(\overline{x}_1)}  $, there holds
\[|y-\overline{x}_1|\ge \frac 14|\overline{x}_1-\overline{x}_i|.  \]
It's easy to get
\begin{align}  \label{A.16}
\int_{\mathbb{R}^N \setminus {B_{\frac{\overline  d_i} 4}(\overline{x}_1)}} \frac{\Lambda^{\frac{N+2} 2}}{(1+\Lambda^2|y-\overline{x}_1|^{2})^{\frac{N+2} 2}}  \frac{\Lambda^{\frac{N-2}{2}}}{(1+\Lambda^2|y-\overline{x}_i|^{2})^{\frac{N-2}{2}}}
= O\Big(\frac 1 {|\overline{x}_1-\overline{x}_i|^{N-\epsilon_0}}\Big).
\end{align}
Combining \eqref{A10}, \eqref{A.15} and \eqref{A.16}, we can get
\begin{align}
\int_{\mathbb{R}^N}  U_{\overline{x}_1, \Lambda}^{2^*-1} U_{\overline{x}_i, \Lambda}
 =\frac{B_0}{\Lambda^{N-2}|\overline{x}_1-\overline{x}_i|^{N-2}}+O\Big(\frac 1 {|\overline{x}_1-\overline{x}_i|^{N-\epsilon_0}}\Big).
\end{align}

Similarly, we can get
\begin{align*}
\int_{\R^N}  U_{\overline{x}_1, \Lambda}^{2^*-1}  U_{\underline{x}_i, \Lambda}
=  \frac{B_0}{\Lambda^{N-2}|\overline{x}_1-\underline{x}_i|^{N-2}}+O\Big(\frac 1 {|\overline{x}_1-\underline{x}_i|^{N-\epsilon_0}}\Big),
\end{align*}
 $\text{for} ~ i=1,\cdots, k.  $

 \end{proof}

 \begin{lemma}  \label{express9}
 Suppose that  $ K(|y|)  $ satisfies   ${\bf H} $ and  $N\ge 5 $, $(r,h,\Lambda) \in{{\mathscr  S}_k}$.
 We have the expansion for $k \to \infty$
\begin{align*}
I(W_{r,h,\Lambda})
\,=\,& kA_1-k \int_{\R^N}U_{\overline{x}_1, \Lambda}^{2^*-1}\Big(\sum_{j=2}^k  U_{\overline{x}_j, \Lambda}+\sum_{j=1}^k U_{\underline{x}_j, \Lambda}\Big)
\\[2mm]
&+k\Big[\frac{A_2}{\Lambda^{m}{\bf r}^{m}}
+\frac{A_3}{\Lambda^{m-2}{\bf r}^{m}}({\bf r}-r)^2\Big]
+k\frac{\mathcal {C}(r, \Lambda)}{{\bf r}^{m}}  \big({\bf r}-r\big)^{2+\sigma}
\\[2mm]
&+k\frac{\mathcal {C}(r, \Lambda)}{{\bf r}^{m+\sigma}}+k O\Big(\Big(\frac{k}{{\bf r}}\Big)^{N-\epsilon_0}\Big)
+k O\Big(\frac{1}{k^{m}} \Big(\frac{k}{{\bf r}}\Big)^{N-2}\Big),
\end{align*}
where $\mathcal {C}(r, \Lambda)$ denotes function independent of $h$ and should be order of $O(1)$,
\begin{equation}\label{A1A2}
A_1=\Big(1-\frac{2}{2^*}\Big) \int_{\R^N}|U_{0,1}|^{2^*},
\quad
A_2=\frac{2 c_0}{2^*} \int_{\R^N}|y_1|^{m}U_{0,1}^{2^*},
\end{equation}
\begin{equation}\label{B_0B_1}
A_3=\frac{c_0 m(m-1)}{2^*}\int_{\R^N}|y_1|^{m-2}U_{0,1}^{2^*},
\end{equation}
and  $ \epsilon_0  $ is constant can be chosen small enough.
 \end{lemma}

  \begin{proof}
Recalling  the definition of  $I(u) $ as in \eqref{energyfunctial}, then we obtain that
\begin{align}
 I(W_{r,h,\Lambda})
 \,=\,& \frac{1}{2} \int_{\mathbb{R}^N}|\nabla W_{r,h,\Lambda}|^2- \frac{1}{2^*} \int_{\mathbb{R}^N} K\Big(\frac{|y|}{{\bf r}}\Big) W_{r,h,\Lambda}^{2^*}
\nonumber\\[2mm]
:=\,&I_1-I_2.
\end{align}

According to the expression of $W_{r,h,\Lambda}$, we have
\begin{align}\label{estimateofI1}
I_1
\,=\,&\frac{1}{2}\sum_{j=1}^k \sum_{i=1}^k \int_{\mathbb{R}^N} -\Delta \Big(U_{\overline{x}_j, \Lambda}+U_{\underline{x}_j, \Lambda}\Big) \Big(U_{\overline{x}_i, \Lambda}+U_{\underline{x}_i, \Lambda}\Big)
\nonumber\\[2mm]
\,=\,&k \sum_{j=1}^k \int_{\mathbb{R}^N} \Big(U_{\overline{x}_1,\Lambda}^{2^*-1}U_{\overline{x}_j, \Lambda}+U_{\underline{x}_1, \Lambda}^{2^*-1} U_{\overline{x}_j, \Lambda}\Big)
\nonumber\\[2mm]
\,=\,&k\int_{\mathbb{R}^N} \Big(U_{0,1}^{2^*}+\sum_{j=2}^kU_{\overline{x}_1, \Lambda}^{2^*-1} U_{\overline{x}_j,\Lambda}\Big)
\,+\,k\int_{\mathbb{R}^N}\sum_{j=1}^k U_{\underline{x}_1,\Lambda}^{2^*-1} U_{\overline{x}_j, \Lambda}
\nonumber\\[2mm]
\,=\,& k\int_{\mathbb{R}^N}U_{0,1}^{2^*} +k\int_{\mathbb{R}^N} U_{\overline{x}_1, \Lambda}^{2^*-1}\Big(\sum_{j=2}^kU_{\overline{x}_j, \Lambda}\,+\,\sum_{i=1}^k  U_{\underline{x}_j, \Lambda}\Big).
\end{align}

For  $I_2$,  using the symmetry of function  $W_{r,h,\Lambda}$,  we have
 \begin{align}\label{estimateofI2}
 I_2
 \,= \, &\frac{2k}{2^*} \int_{\Omega_1^{+}} K\Big(\frac{|y|}{{\bf r}}\Big)  W_{r,h,\Lambda}^{2^*}
\nonumber\\[2mm]
\,= \, &  \frac{2k}{2^*} \int_{\Omega_1^{+}} K\Big(\frac{|y|}{{\bf r}}\Big)  \Bigg\{U_{\overline{x}_1, \Lambda}^{2^*}+2^* U_{\overline{x}_1, \Lambda}^{2^*-1}  \Big(\sum_{j=2}^k U_{\overline{x}_j, \Lambda}
+\sum_{j=1}^k U_{\underline{x}_j, \Lambda}\Big)
 \nonumber\\[2mm]
 & \qquad\qquad\qquad \quad + O\Big(  U_{\overline{x}_1, \Lambda}^{\frac{2^*} 2} \big(\sum_{j=2}^k  U_{\overline{x}_j, \Lambda}
+\sum_{j=1}^k  U_{\underline{x}_j, \Lambda}\big)^{\frac{2^*} 2} \Big) \Bigg\}
\nonumber\\
\,:=\,&\frac{2k}{2^*} \Big(I_{21}+I_{22}+I_{23}\Big).
 \end{align}

 For  $ y\in \Omega_1^{+}  $,
from Lemma \ref{b.0}, we have
\begin{align*}
  \Big(\sum_{j=2}^k  U_{\overline{x}_j, \Lambda}+\sum_{j=1}^k U_{\underline{x}_j, \Lambda}\Big)
  \le  \frac{C}{\big(1+|y-\overline{x}_{1}|\big)^{\frac{(N-2) \epsilon_0}{N}}} \Big(\frac{k}{{\bf r}}\Big)^{N-2-\frac{(N-2) \epsilon_0}{N}},
\end{align*}
 with  $ \epsilon_0 > 0  $ can be chosen small enough.
Then we can get
\begin{align}\label{estimateofI23}
 I_{23}\,=\,&  O\Big(   \int_{\Omega_1^{+}}K\Big(\frac{|y|}{{\bf r}}\Big) U_{\overline{x}_1, \Lambda}^{\frac{2^*} 2} \big(\sum_{j=2}^k  U_{\overline{x}_j, \Lambda}
+\sum_{j=1}^k U_{\underline{x}_j, \Lambda}\big)^{\frac{2^*} 2}  \Big)= O\Big(\Big(\frac{k}{{\bf r}}\Big)^{N-\epsilon_0}\Big).
 \end{align}

For  $I_{21} $, we can rewrite it as following
\begin{align*}
I_{21}
&=  \int_{\Omega_1^{+}} U_{\overline{x}_1, \Lambda}^{2^*}+\int_{\Omega_1^{+}}  \Big[K\Big(\frac{|y|}{{\bf r}}\Big)-1\Big]  U_{\overline{x}_1, \Lambda}^{2^*}
\nonumber\\[2mm]
&= \int_{\mathbb{R}^N} U_{0,1}^{2^*}+\int_{\Omega_1^{+}}  \Big[K\Big(\frac{|y|}{{\bf r}}\Big)-1\Big]  U_{\overline{x}_1, \Lambda}^{2^*}+O \Big(\Big(\frac{k}{{\bf r}}\Big)^{N}\Big).
\end{align*}
 Furthermore, we  obtain
\begin{align*}
 \int_{\Omega_1^{+}}  \Big[K\Big(\frac{|y|}{{\bf r}}\Big)-1\Big]  U_{\overline{x}_1, \Lambda}^{2^*}
=& \int_{\Omega_1^{+} \cap \big\{y :\,|\frac{|y|}{\bf r}-1|\ge \delta  \big\}}\Big[K\Big(\frac{|y|}{{\bf r}}\Big)-1\Big]  U_{\overline{x}_1, \Lambda}^{2^*}
\\[2mm]
&+\int_{\Omega_1^{+} \cap \big\{y :\,|\frac{|y|}{\bf r}-1|\le \delta  \big\}} \Big[K\Big(\frac{|y|}{{\bf r}}\Big)-1\Big]  U_{\overline{x}_1, \Lambda}^{2^*}.
\end{align*}
  When  $|\frac{|y|}{\bf r}-1|\ge \delta  $, there holds
\begin{align*}
|y-\overline{x}_1| \ge \big||y |-{\bf r}\big|\,-\,\big|{\bf r} -|\overline{x}_1|\big|\ge \frac{1}{2} \delta {\bf r}.
\end{align*}
 Thus we can easily get
\begin{align*}
 \int_{\Omega_1^{+} \cap \big\{y :\,|\frac{|y|}{\bf r}-1|\ge \delta  \big\}} \Big[K\Big(\frac{|y|}{{\bf r}}\Big)-1\Big]  U_{\overline{x}_1, \Lambda}^{2^*} \le  \frac{C}{{\bf r}^{N-\epsilon_0}}.
\end{align*}
 If  $|\frac{|y|}{\bf r}-1|\le \delta  $,  recalling the decay property of  $ K $,  we can obtain that
\begin{align*}
&\int_{\Omega_1^{+} \cap \big\{y :\,|\frac{|y|}{\bf r}-1|\le \delta  \big\}} \Big[K\Big(\frac{|y|}{{\bf r}}\Big)-1\Big]  U_{\overline{x}_1, \Lambda}^{2^*}
\\[2mm]
&= -c_0 \frac{1}{{\bf r}^{m}} \int_{\Omega_1^{+} \cap \big\{y :\,|\frac{|y|}{\bf r}-1|\le \delta  \big\}}\big||y|-{\bf r} \big|^{m} \, U_{\overline{x}_1, \Lambda}^{2^*}
\\[2mm]
& \quad+O\Big(\frac{1}{{\bf r}^{m+\sigma} }  \int_{\Omega_1^{+} \cap \big\{y :\,|\frac{|y|}{\bf r}-1|\le \delta  \big\}}\big||y|-{\bf r} \big|^{m+\sigma} \, U_{\overline{x}_1, \Lambda}^{2^*} \Big)
\\[2mm]
&= -c_0 \frac{1}{{\bf r}^{m}} \int_{\R^N}\big||y|-{\bf r} \big|^{m} \, U_{\overline{x}_1, \Lambda}^{2^*}
+O \Big(\int_{\R^N \setminus B_{\frac{{\bf r}}{k}}(\overline{x}_1)}  \Big(\frac{|y|^{m}}{{\bf r}^{m}}+1\Big)  U_{\overline{x}_1,\Lambda}^{2^*}\Big)
 \\[2mm]
& \quad+O\Big(\frac{1}{{\bf r}^{m+\sigma} } \int_{\Omega_1^{+} \cap \big\{y :\,|\frac{|y|}{\bf r}-1|\le \delta  \big\}}\big||y|-{\bf r} \big|^{m+\sigma} \, U_{\overline{x}_1, \Lambda}^{2^*} \Big)
\\[2mm] &
= -c_0 \frac{1}{{\bf r}^{m}} \int_{\R^N}
 \big||y+\overline{x}_1|-{\bf r} \big|^{m} \, U_{0, \Lambda}^{2^*}
\\[2mm]
& \quad+O\Big(\frac{1}{{\bf r}^{m+\sigma} } \int_{\Omega_1^{+} \cap \big\{y :\,|\frac{|y|}{\bf r}-1|\le \delta  \big\}}\big||y|-{\bf r} \big|^{m+\sigma} \, U_{\overline{x}_1, \Lambda}^{2^*} \Big)+O  \Big(\Big(\frac{k}{{\bf r}}\Big)^{N-\epsilon_0}\Big).
\end{align*}
Furthermore, recalling  $|\overline{x}_1|= r $ and using the symmetry property,  we have
\begin{align*}
 \int_{\R^N} \big||y+\overline{x}_1|-{\bf r} \big|^{m}U_{0,\Lambda}^{2^*}
 =  \int_{\R^N} \big||y+e_1 r|-{\bf r} \big|^{m}U_{0,\Lambda}^{2^*},
\end{align*}
where  $ e_1=(1, 0,\cdots, 0) $.

We  get
\begin{align*}
&\int_{\R^N}||y+\overline{x}_1|-{\bf r}|^{m}U_{0,\Lambda}^{2^*}\\[2mm]
=&\int_{\R^N}|y_1|^{m}U_{0,\Lambda}^{2^*}+\frac12 m(m-1) \int_{\R^N}|y_1|^{m-2}U_{0,\Lambda}^{2^*}({\bf r} -r)^2
+\mathcal {C}(r, \Lambda)({\bf r} -r)^{2+\sigma},
\end{align*}
here $\mathcal {C}(r, \Lambda)$ denote functions which are independent of  $h$ and can be absorbed in  $O(1) $.

Similarly, we can also have the following expression
\begin{align*}
&O\Big(\frac{1}{{\bf r}^{m+\sigma} } \int_{\Omega_1^{+} \cap \big\{y :\,|\frac{|y|}{\bf r}-1|\le \delta  \big\}}\big||y|-{\bf r} \big|^{m+\sigma} \, U_{\overline{x}_1, \Lambda}^{2^*} \Big)
\\[2mm]
&=O\Big(\frac{1}{{\bf r}^{m+\sigma} } \int_{\R^N}\big||y|-{\bf r} \big|^{m+\sigma} \, U_{\overline{x}_1, \Lambda}^{2^*} \Big)+O  \Big(\Big(\frac{k}{{\bf r}}\Big)^{N-\epsilon_0}\Big)
\\[2mm]
&=  \frac{\mathcal {C}(r, \Lambda)}{{\bf r}^{m+\sigma}}
\,+\,O  \Big(\Big(\frac{k}{{\bf r}}\Big)^{N-\epsilon_0}\Big).
\end{align*}

 Then, we can obtain that
\begin{align}\label{estimateofI21}
I_{21}
= &
\int_{\R^N}|U_{0,1}|^{2^*}-\frac{c_0}{\Lambda^{m}{\bf r}^{m}}
\int_{\R^N}|y_1|^{m}U_{0,1}^{2^*}
\nonumber\\[2mm]
&-\frac12 m(m-1)\frac{c_0}{\Lambda^{m-2}{\bf r}^{m}}
 \int_{\R^N}|y_1|^{m-2}U_{0,1}^{2^*}
({\bf r} -r)^2
\nonumber\\[2mm]
&+\frac{\mathcal {C}(r, \Lambda)}{{\bf r}^{m}}({\bf r} -r)^{2+\sigma}
+\frac{\mathcal {C}(r, \Lambda)}{{\bf r}^{m+\sigma}}
+O\Big(\Big(\frac{k}{{\bf r}}\Big)^{N-\epsilon_0}\Big).
\end{align}

Finally, we consider  $I_{22} $
\begin{align*}
I_{22}
\,=\,&2^{*} \int_{\Omega_1^{+}}  U_{\overline{x}_1, \Lambda}^{2^*-1}  \Big(\sum_{j=2}^k  U_{\overline{x}_j, \Lambda}+\sum_{j=1}^k U_{\underline{x}_j, \Lambda}\Big)
\\
&+2^{*}\int_{\Omega_1^{+}} \Big[K\Big(\frac{|y|}{{\bf r}}\Big) -1\Big]  U_{\overline{x}_1, \Lambda}^{2^*-1}  \Big(\sum_{j=2}^k  U_{\overline{x}_j, \Lambda}+\sum_{j=1}^k U_{\underline{x}_j, \Lambda}\Big)
\\
\,=\,&2^{*} \int_{\R^N }  U_{\overline{x}_1, \Lambda}^{2^*-1}  \Big(\sum_{j=2}^k  U_{\overline{x}_j, \Lambda}+\sum_{j=1}^k U_{\underline{x}_j, \Lambda}\Big)
-2^{*}\int_{\R^N \setminus \Omega_1^{+}} U_{\overline{x}_1, \Lambda}^{2^*-1}  \Big(\sum_{j=2}^k  U_{\overline{x}_j, \Lambda}+\sum_{j=1}^k U_{\underline{x}_j, \Lambda}\Big)
\\
&+2^{*}\int_{\Omega_1^{+}} \Big[K\Big(\frac{|y|}{{\bf r}}\Big) -1\Big]  U_{\overline{x}_1, \Lambda}^{2^*-1}  \Big(\sum_{j=2}^k  U_{\overline{x}_j, \Lambda}+\sum_{j=1}^k U_{\underline{x}_j, \Lambda}\Big)
\\
:\,=\,&I_{221}+I_{222}+I_{223}.
\end{align*}
For  $I_{222} $, it is easy to derive that
\begin{align*}
  \sum_{j=1}^k  \int_{\R^N \setminus \Omega_1^{+}} U_{\overline{x}_1, \Lambda}^{2^*-1}  U_{\underline{x}_j, \Lambda}=  O\Big(\frac{k^N}{{\bf r}^N}\Big).
\end{align*}
 Moreover, we know that
\begin{align*}
 &\sum_{j=2}^k \int_{\R^N \setminus \Omega_1^{+}}  U_{\overline{x}_1, \Lambda}^{2^*-1}  U_{\overline{x}_j, \Lambda}
 \\
\,=\,& \sum_{j=2}^k \int_{\big(\R^N \setminus \Omega_1^{+}\big) \cap  B_{\overline{d}_j/2} (\overline{x}_1)} U_{\overline{x}_1, \Lambda}^{2^*-1} U_{\overline{x}_j, \Lambda}
 +\sum_{j=2}^k \int_{\big(\R^N \setminus \Omega_1^{+}\big) \setminus  B_{\overline{d}_j/2} (\overline{x}_1)} U_{\overline{x}_1, \Lambda}^{2^*-1} U_{\overline{x}_j, \Lambda}
 \\
\,=\,& \sum_{j=2}^k\int_{\big(\R^N \setminus \Omega_1^{+}\big) \cap  B_{\overline{d}_j/2} (\overline{x}_1)} U_{\overline{x}_1, \Lambda}^{2^*-1} U_{\overline{x}_j, \Lambda}
 +O\Big(\sum_{j=2}^k\frac 1 {|\overline{x}_1-\overline{x}_j|^{N-\epsilon_0}}\Big)
\\[2mm]
\,\le\,& C \sum_{j=2}^k \int_{B_{\overline{d}_j/2} (\overline{x}_1) \setminus  B_{\overline{d}_2/2} (\overline{x}_1)} U_{\overline{x}_1, \Lambda}^{2^*-1} U_{\overline{x}_j, \Lambda}
+O\Big(\sum_{j=2}^k\frac 1 {|\overline{x}_1-\overline{x}_j|^{N-\epsilon_0}}\Big)
\\[2mm]
\,=\,&C \sum_{j=2}^k\frac 1 {|\overline{x}_1-\overline{x}_j|^{N-2}}  \int_{B_{\Lambda \overline{d}_j/2} (0) \setminus  B_{\Lambda \overline{d}_2/2} (0)}  \frac{1}{(1+z^{2})^{\frac{N+2} 2}}
+O\Big(\sum_{j=2}^k\frac 1 {|\overline{x}_1-\overline{x}_j|^{N-\epsilon_0}}\Big)
\\[2mm]
\,\le\, &C  \sum_{j=2}^k\frac 1 {|\overline{x}_1-\overline{x}_j|^{N-2}}  O\Big(\frac 1 {\overline{d}_ 2^2}\Big)
+O\Big(\sum_{j=2}^k\frac 1 {|\overline{x}_1-\overline{x}_j|^{N-\epsilon_0}}\Big)
\\
\,=\,&O\Big(\frac{k^2}{r^2}\Big)  \sum_{j=2}^k \frac 1 {|\overline{x}_1-\overline{x}_j|^{N-2}}+O\Big(\sum_{j=2}^k  \frac 1 {|\overline{x}_1-\overline{x}_j|^{N-\epsilon_0}}\Big)
 \\[2mm]
\,=\,& O\Big(\frac{k^N}{{\bf r}^N}\Big)+O\Big(\sum_{j=2}^k \frac 1 {|\overline{x}_1-\overline{x}_j|^{N-\epsilon_0}}\Big),
\end{align*}
where  $ \overline{d}_j =|\overline{x}_1-\overline{x}_j|$ for $j=2, \cdots, k$ and $ \overline{d}_ 2=|\overline{x}_1-\overline{x}_2|= 2r \sqrt{1-h^2} \sin{\frac{\pi}{k}}=  O\big(\frac r k\big) $.
 Then we get
\begin{align}\label{estimateofI222}
 I_{222}= O\Big(\Big(\frac{k}{{\bf r}}\Big)^{{N-\epsilon_0}}\Big).
 \end{align}
Next, we consider the term $I_{223}$. In fact, we have
\begin{align*}
I_{223}
=&\int_{\Omega_1^{+} \cap \big\{y :\,|\frac{|y|}{\bf r}-1|\ge \delta  \big\}}\Big[K\Big(\frac{|y|}{{\bf r}}\Big) -1\Big]  U_{\overline{x}_1, \Lambda}^{2^*-1}  \Big(\sum_{j=2}^k  U_{\overline{x}_j, \Lambda}+\sum_{j=1}^k U_{\underline{x}_j, \Lambda}\Big)
\\[2mm]
&+\int_{\Omega_1^{+} \cap \big\{y :\,|\frac{|y|}{\bf r}-1|\le \delta  \big\}}\Big[K\Big(\frac{|y|}{{\bf r}}\Big) -1\Big]  U_{\overline{x}_1, \Lambda}^{2^*-1}  \Big(\sum_{j=2}^k  U_{\overline{x}_j, \Lambda}+\sum_{j=1}^k U_{\underline{x}_j, \Lambda}\Big).
\end{align*}
When  $|\frac{|y|}{\bf r}-1|\ge \delta  $, there hold
\begin{align*}
|y-\overline{x}_1| \ge \big||y|-{\bf r}\big|-\big|{\bf r} -|\overline{x}_1|\big|
\ge \frac{1}{2} \delta {\bf r}.
\end{align*}
And for $ y\in \Omega_1^{+}$ and $|\frac{|y|}{\bf r}-1|\ge \delta$, we have
\begin{align}\label{estialpha0}
\Big(\sum_{j=2}^kU_{\overline{x}_j,\Lambda}+\sum_{j=1}^k U_{\underline{x}_j,\Lambda}\Big) & \le C \Big(\frac{k}{{\bf r}}\Big)^\alpha\frac{1}{\big(1+|y-\overline{x}_{1}|\big)^{N-2-\alpha}},
\end{align}
 with $ \alpha=(\frac{N-2-m}{N-2},\frac{N-2}{2})$. Then we can get easily
\begin{align*}
& \int_{\Omega_1^{+} \cap \big\{y : \,|\frac{|y|}{\bf r}-1|\ge \delta \big\}}\Big[K\Big(\frac{|y|}{{\bf r}}\Big) -1\Big]  U_{\overline{x}_1, \Lambda}^{2^*-1}  \Big(\sum_{j=2}^k  U_{\overline{x}_j, \Lambda}+\sum_{j=1}^k U_{\underline{x}_j, \Lambda}\Big)
\\[2mm]
&\le \frac{C}{{\bf r}^{N-\alpha-\epsilon_0}} \Big(\frac{k}{{\bf r}}\Big)^\alpha
\le C\Big(\frac{k}{{\bf r}}\Big)^{N}.
\end{align*}
If  $|\frac{|y|}{\bf r}-1|\le \delta$, then
\begin{align*}
&\int_{\Omega_1^{+} \cap \big\{y : \,|\frac{|y|}{\bf r}-1|\le \delta  \big\}}  \Big[K\Big(\frac{|y|}{{\bf r}}\Big) -1\Big]  U_{\overline{x}_1, \Lambda}^{2^*-1}  \Big(\sum_{j=2}^k  U_{\overline{x}_j, \Lambda}+\sum_{j=1}^k U_{\underline{x}_j, \Lambda}\Big)
\\[2mm]
&\le \frac{C}{{\bf r}^{m}}  \int_{\Omega_1^{+} \cap \big\{y : \,|\frac{|y|}{\bf r}-1|\le \delta \big\}} \big|| y|-{\bf r} \big|^{m}  U_{\overline{x}_1, \Lambda}^{2^*-1}  \Big(\sum_{j=2}^k  U_{\overline{x}_j, \Lambda}+\sum_{j=1}^k U_{\underline{x}_j, \Lambda}\Big)
\\[2mm]
&=\frac{C}{{\bf r}^{m}}  \int_{\Omega_1^{+} \cap \big\{y : \,|\frac{|y|}{\bf r}-1|\le \delta\big\} \cap\big\{y : \,|y- \overline{x}_1|\le \frac{\delta_1{\bf r}} k\big\}} \big|| y|-{\bf r} \big|^{m}U_{\overline{x}_1, \Lambda}^{2^*-1}  \Big(\sum_{j=2}^k  U_{\overline{x}_j, \Lambda}+\sum_{j=1}^k U_{\underline{x}_j, \Lambda}\Big)
\\[2mm]
&\quad+\frac{C}{{\bf r}^{m}}\int_{\Omega_1^{+} \cap \big\{y : \,|\frac{|y|}{\bf r}-1|\le \delta\big\} \cap\big\{y : \,|y- \overline{x}_1|\ge  \frac{\delta_1{\bf r}} k\big\}}\big|| y|-{\bf r} \big|^{m}  U_{\overline{x}_1, \Lambda}^{2^*-1}  \Big(\sum_{j=2}^k  U_{\overline{x}_j, \Lambda}+\sum_{j=1}^k U_{\underline{x}_j, \Lambda}\Big),
\end{align*}
  where  $ \delta_1 $ is small constant.
If  $|y- \overline{x}_1|\le \frac{\delta_1{\bf r}} k $, it is easy to derive
\begin{equation*}
 \big|| y|-{\bf r} \big| \le|y-\overline{x}_1|+||\overline{x}_1|-{\bf r}|\le \frac{\delta_2{\bf r}} k,
 \end{equation*}
 for some small  $ \delta_2 $. Therefore,
\begin{align*}
 \frac{C}{{\bf r}^{m}} \big|| y|-{\bf r} \big|^{m} \le  \frac{C}{k^{m}}.
\end{align*}
 Hence
\begin{align*}
 &\frac{C}{{\bf r}^{m}}  \int_{\Omega_1^{+} \cap \big\{y : \,|\frac{|y|}{\bf r}-1|\le \delta  \big\} \cap  \big\{y : \,|y- \overline{x}_1|\le \frac{\delta_1{\bf r}} k \big\}} \big|| y|-{\bf r} \big|^{m}  U_{\overline{x}_1, \Lambda}^{2^*-1}  \Big(\sum_{j=2}^k  U_{\overline{x}_j, \Lambda}+\sum_{j=1}^k U_{\underline{x}_j, \Lambda}\Big)
 \\[2mm]
 & \le  \frac{C}{k^{m}}  \int_{\R^N} U_{\overline{x}_1, \Lambda}^{2^*-1}  \Big(\sum_{j=2}^k  U_{\overline{x}_j, \Lambda}+\sum_{j=1}^k U_{\underline{x}_j, \Lambda}\Big)
 \\[2mm]
 & \le \frac{C}{k^{m}} \Big(\frac{k}{{\bf r}}\Big)^{N-2}.
\end{align*}
 When $|y- \overline{x}_1|\ge  \frac{\delta_1{\bf r}} k $, combing \eqref{estialpha0}, we can get easily,
\begin{align*}
& \frac{C}{{\bf r}^{m}}  \int_{\Omega_1^{+} \cap \big\{y : \,|\frac{|y|}{\bf r}-1|\le \delta  \big\} \cap  \big\{y : \,|y- \overline{x}_1|\ge  \frac{\delta_1{\bf r}} k\big\}}
\big|| y|-{\bf r} \big|^{m}  U_{\overline{x}_1, \Lambda}^{2^*-1}  \Big(\sum_{j=2}^k  U_{\overline{x}_j, \Lambda}+\sum_{j=1}^k U_{\underline{x}_j, \Lambda}\Big)
\\[2mm]
&\le C\Big(\frac{k}{{\bf r}}\Big)^{N-\epsilon_0}.
\end{align*}
  Thus we can get
  \begin{align}\label{estimateofI223}
 I_{223}= O\Big(\Big(\frac{k}{{\bf r}}\Big)^{N-\epsilon_0}\Big)+O \Big(\frac{1}{k^{m}} \Big(\frac{k}{{\bf r}}\Big)^{N-2}\Big).
  \end{align}

 Combining \eqref{estimateofI1}, \eqref{estimateofI2}, \eqref{estimateofI21}, \eqref{estimateofI23}, \eqref{estimateofI222} and \eqref{estimateofI223}, we can get
\begin{align*}
I(W_{r,h,\Lambda})
\,=\,&k \Big(1-\frac{2}{2^*}\Big) \int_{\R^N}|U_{0,1}|^{2^*}-k \int_{\R^N}  U_{\overline{x}_1, \Lambda}^{2^*-1}  \Big(\sum_{j=2}^k  U_{\overline{x}_j, \Lambda}+\sum_{j=1}^k U_{\underline{x}_j, \Lambda}\Big)
\\[2mm]
&+\frac{2k}{2^*} \Big[\frac{c_0}{\Lambda^{m}{\bf r}^{m}}
\int_{\R^N}|y_1|^{m}U_{0,1}^{2^*}+\frac{c_0 m(m-1)}{2 \Lambda^{m-2}{\bf r}^{m}}
 \int_{\R^N}|y_1|^{m-2}U_{0,1}^{2^*}({\bf r}-r)^2\Big]
\\[2mm]
&+k\frac{\mathcal {C}(r, \Lambda)}{k^{\frac{m(N-2)}{N-2-m}}}({\bf r} -r)^{2+\sigma}
+k\frac{\mathcal {C}(r, \Lambda)}{{\bf r}^{m+\sigma}}+k O\Big(\Big(\frac{k}{{\bf r}}\Big)^{N-\epsilon_0}\Big)+k O\Big(\frac{1}{k^{m}} \Big(\frac{k}{{\bf r}}\Big)^{N-2}\Big).
\end{align*}

\end{proof}

 \medskip
 Combining Lemma \ref{express7}-\ref{express9}, we can get
 the following Proposition which gives the expression of  $ I(W_{r,h,\Lambda}).  $
 \medskip

\begin{proposition} \label{func}
Suppose that  $ K(|y|)  $ satisfies  ${\bf H} $ and  $N\ge 5 $, $(r,h,\Lambda) \in{{\mathscr  S}_k}$. Then we have
\begin{align} \label{express5}
 I(W_{r,h,\Lambda})
\,=\,& kA_1 -\frac{k}{\Lambda^{N-2}}\Big[\,\frac{B_4 k^{N-2}}{(r \sqrt{1-h^2})^{N-2}}\,+\, \frac{B_5 k}{r^{N-2} h^{N-3} \sqrt {1-h^2}}\,\Big]
\nonumber\\[2mm]
&+k\Big[\frac{A_2}{\Lambda^{m}  k^{\frac{(N-2)m}{N-2-m}}}
+\frac{A_3}{\Lambda^{m-2}  k^{\frac{(N-2)m}{N-2-m}}}({\bf r}-r)^2\Big]
+k\frac{\mathcal {C}(r, \Lambda)}{k^{\frac{m(N-2)}{N-2-m}}}({\bf r} -r)^{2+\sigma}
\nonumber\\[2mm]
&+k\frac{\mathcal {C}(r, \Lambda)}{k^{\frac{m(N-2)}{N-2-m}+\sigma}}
+kO\Big(\frac{1}{k^{\big(\frac{m(N-2)}{N-2-m}+\frac{2(N-3)}{N-1}+\sigma\big)}}\Big),
\end{align}
as $k \to \infty$, where $ A_i,(i=1,2,3), B_4, B_5  $ are positive constants.
 \end{proposition}

\begin{proof}
 A direct result of Lemma \ref{express7}-\ref{express9} is
\begin{align*}
 I(W_{r,h,\Lambda})
\,=\,& k A_1 \,-\,\frac{k}{\Lambda^{N-2}} \Big[\,\frac{B_4k^{N-2}}{(r \sqrt{1-h^2})^{N-2}}\,+\,\frac{B_5k}{r^{N-2} h^{N-3} \sqrt {1-h^2}}\,\Big]
\nonumber\\[2mm]
&\,+\,k\Big[\frac{A_2}{\Lambda^{m}{\bf r}^{m}}\,+\,\frac{A_3}{\Lambda^{m-2}{\bf r}^{m}}({\bf r}-r)^2\Big]
\,+\,k\frac{\mathcal {C}(r,\Lambda)}{k^{\frac{m(N-2)}{N-2-m}}}({\bf r} -r)^{2+\sigma}
\\[2mm]
&\,+\,k\frac{\mathcal {C}(r,\Lambda)}{k^{\frac{(N-2)m}{N-2-m}+\sigma}}
\,+\,kO\Big(\Big(\frac{k}{{\bf r}}\Big)^{N-\epsilon_0}\Big)
\,+\,kO\Big(\frac{1}{k^{m}} \Big(\frac{k}{{\bf r}}\Big)^{N-2}\Big)
\nonumber \\[2mm]
&\,+\,kO\Big(\frac{\sigma_1(k) k^{N-2}}{\big(r \sqrt{1-h^2}\big)^{N-2}}\Big)
+k O\Big(\frac{ \sigma_2(k)k}{r^{N-2} h^{N-3} \sqrt {1-h^2}}\Big),
\end{align*}
 with $B_4=B_0B_1, B_5=B_0B_2$ are positive constants. From the expressions of  $ \sigma_1(k), \sigma_2(k)  $ and asymptotic expression  of  $h, r $ as in \eqref{sigmak}, \eqref{definitionofsk}  $,  $ we can show that
 \[ \frac{\sigma_1(k) k^{N-2}}{\big(r \sqrt{1-h^2}\big)^{N-2}},
 \quad \frac{\sigma_2(k)k}{r^{N-2} h^{N-3} \sqrt {1-h^2}} , \]
can be absorbed in  $ O\Big(\frac{1}{k^{\big(\frac{m(N-2)}{N-2-m}+\frac{2(N-3)}{N-1}+\sigma\big)}}\Big). $

\medskip
Noting that
 $ m > \frac{N-2}{2} $ implies
\[  \frac{N-3}{N-1} < \frac{m}{N-2-m},  \]
thus provided with  $ \epsilon_0, \sigma$ small enough, we can get
\[\Big(\frac{k}{{\bf r}}\Big)^{N-\epsilon_0}
= \frac 1 {k^{\frac{m(2-\epsilon_0)}{N-2-m}}} \frac 1 {k^{\frac{m(N-2)}{N-2-m}}}
\le C \frac{1}{k^{\big(\frac{m(N-2)}{N-2-m}+\frac{2(N-3)}{N-1}+\sigma\big)}}. \]
Since  $ m\ge 2 $, we can check that
\[ \frac{1}{k^{m}} \Big(\frac{k}{{\bf r}}\Big)^{N-2}  \le \frac{C}{k^{\big(\frac{m(N-2)}{N-2-m}+\frac{2(N-3)}{N-1}+\sigma\big)}}.\]
Thus we can get \eqref{express5}.
\end{proof}

\medskip
 To get the expansions of $\frac{F(r,h,\Lambda)}{\partial \Lambda}, \frac{F(r,h,\Lambda)}{\partial h}$, we need the following expansions  for $ \frac{\partial I(W_{r,h,\Lambda})}{\partial \Lambda},  \frac{\partial I(W_{r,h,\Lambda})}{\partial h}$.\begin{proposition}\label{func2}
Suppose that  $ K(|y|)  $ satisfies   ${\bf H} $ and  $N\ge 5 $, $(r,h,\Lambda) \in{{\mathscr  S}_k}$.
We have
\begin{align*}
 \frac{\partial I(W_{r,h,\Lambda})}{\partial \Lambda}
\,=\,&\frac{k (N-2)}{\Lambda^{N-1}} \Big[\frac{B_4 k^{N-2}}{(r \sqrt{1-h^2})^{N-2}}+\frac{B_5 k}{r^{N-2} h^{N-3} \sqrt {1-h^2}}\,\Big] \nonumber
\\[2mm]
&-k \Big[\,\frac{m  A_2}{\Lambda^{m+1}k^{\frac{(N-2)m}{N-2-m}}}+\frac{(m-2) A_3}{\Lambda^{m-1} k^{\frac{(N-2)m}{N-2-m}}} ({\bf r}-r)^2\,\Big]
+kO \Big(\frac 1 {k^{\frac{(N-2)m}{N-2-m}+\sigma}}\Big),
\end{align*}
as $k \to \infty$, where the constants  $B_i, i= 4,5 $ and  $A_i, i=2,3 $ are  defined in Proposition \ref{func}.

\end{proposition}

\begin{proof} The proof of this proposition is standard and the reader can refer to \cite{WY-10-JFA} for details.
\end{proof}

 \medskip
  \begin{proposition}\label{func3}
Suppose that  $ K(|y|)  $ satisfies  ${\bf H} $ and  $N\ge 5 $, $(r,h,\Lambda) \in{{\mathscr  S}_k}$. Then we have
\begin{align} \label{express6}
\frac{\partial I(W_{r,h,\Lambda})}{\partial h}
\,=\,&  -\frac{k}{\Lambda^{N-2}}\Big[\, (N-2) \frac{B_4 k^{N-2}}{ r^{N-2}( \sqrt{1-h^2})^{N}} h -(N-3) \frac{B_5 k}{ r^{N-2}h^{N-2} \sqrt {1-h^2}}  \,\Big]
\nonumber\\[2mm]
& \quad  +kO\Big(\frac{1}{k^{\big(\frac{m(N-2)}{N-2-m}+\frac{(N-3)}{N-1}+\sigma\big)}}\Big)
\end{align}
as $k \to \infty$.
\end{proposition}

\begin{proof}
Recall
\begin{align} \label{estimate kernel}
\overline{\mathbb{Z}}_{2j}
 \le  C  \frac{r} {(1+|y-\overline{x}_j|  )^{N-1}  },
  \quad   \underline{\mathbb{Z}}_{2j}  \le  C  \frac{r} {(1+|y-\underline{x}_j|  )^{N-1}  }.
\end{align}
We know that
\begin{align}\label{express10}
 \frac{\partial I(W_{r,h,\Lambda})}{\partial h}
 \,=\,&  \frac{1}{2}  \frac{\partial } {\partial h } \int_{\mathbb{R}^N}|\nabla W_{r,h,\Lambda}|^2- \frac{1}{2^*}  \frac{\partial } {\partial h}   \int_{\mathbb{R}^N} K\Big(\frac{|y|}{{\bf r}}\Big) W_{r,h,\Lambda}^{2^*}
\nonumber\\[2mm]
=\,&  k \frac{\partial } {\partial h } \int_{\mathbb{R}^N} U_{\overline{x}_1, \Lambda}^{2^*-1}\Big(\sum_{j=2}^kU_{\overline{x}_j, \Lambda}\,+\,\sum_{i=1}^k  U_{\underline{x}_j, \Lambda}\Big)
\nonumber\\[2mm]
\quad &- \int_{\mathbb{R}^N} K\Big(\frac{|y|}{{\bf r}}\Big) W_{r,h,\Lambda}^{2^*-1}    \Big(   \overline{\mathbb{Z}}_{21}
+ \sum_{j=2 }^k  \overline{\mathbb{Z}}_{2j}
+ \sum_{j=1}^k   \underline{\mathbb{Z}}_{2j}   \Big).
\end{align}
From \eqref{express10}, similar to the  calculations in the proof of Proposition \ref{express9},  we can get
\begin{align}   \label{frac h}
   \frac{\partial I(W_{r,h,\Lambda})}{\partial h} =  - k    \frac{\partial }{\partial h}  \int_{\R^N}   U_{\overline{x}_1, \Lambda}^{2^*-1}   \Big(\sum_{j=2}^kU_{\overline{x}_j, \Lambda}\,+\,\sum_{i=1}^k  U_{\underline{x}_j, \Lambda}\Big)  +  k^2O  \Big(\Big(\frac{k}{{\bf r}}\Big)^{N-\epsilon_0}\Big).
  \end{align}

Then by some  tedious but straightforward analysis, we can get
    \begin{align}
\frac{\partial I(W_{r,h,\Lambda})}{\partial h}
\,=\,&  -\frac{k}{\Lambda^{N-2}}\Big[\, (N-2) \frac{B_4 k^{N-2}}{ r^{N-2}( \sqrt{1-h^2})^{N}} h -(N-3) \frac{B_5 k}{ r^{N-2}h^{N-2} \sqrt {1-h^2}}  \,\Big]
\nonumber\\[2mm]
& \quad \quad +h \frac{B_5 k}{r^{N-2} h^{N-3}({1-h^2})^{\frac 32}}  \,\Big]  +   k^2O  \Big(\Big(\frac{k}{{\bf r}}\Big)^{N-\epsilon_0}\Big),
\end{align}
for some $\epsilon_0$ small enough.
In fact, we  know that  $ k  \Big(\frac{k}{{\bf r}}\Big)^{N-\epsilon_0}$ and $  h \frac{B_5 k}{r^{N-2} h^{N-3}({1-h^2})^{\frac 32}} $  can be absorbed in   $O\Big(\frac{1}{k^{\big(\frac{m(N-2)}{N-2-m}+\frac{(N-3)}{N-1}+\sigma\big)}}\Big)$  provided with
$ m $ satisfying \eqref{assumptionform}   and $ \epsilon_0, \sigma$ small enough.   Then  we can get  \eqref{express6} directly.

\end{proof}

 \section{Some basic estimates and lemmas} \label{appendixB}

 \begin{lem} \label{b.0}
  Under the condition  $(r,h,\Lambda) \in {{\mathscr  S}_k} $, for  $ y\in \Omega_1^{+}  $  there exists a constant $C$ such that
\begin{align*}
 \Big(\sum_{j=2}^k  U_{\overline{x}_j, \Lambda}+\sum_{j=1}^k U_{\underline{x}_j, \Lambda}\Big)
  \le  C\Big(\frac{k}{{\bf r}}\Big)^\alpha  \frac{1}{\big(1+|y-\overline{x}_{1}|\big)^{N-2-\alpha}},
\end{align*}
 with  $ \alpha=(1, N-2)  $.
 \end{lem}
 \begin{proof}
 For  $ y\in \Omega_1^{+}  $ and $ j= 2, \cdots, k  $, we have
\begin{align*}
|y-\overline{x}_{j}|\ge|\overline{x}_1-\overline{x}_j|-|y-\overline{x}_1|\ge \frac 14|\overline{x}_1-\overline{x}_j|,  \quad \text {if} ~|y-\overline{x}_1|\le  \frac 14|\overline{x}_1-\overline{x}_j|,
\end{align*}
 and
\begin{align*}
|y-\overline{x}_{j}|\ge|y-\overline{x}_1|\ge \frac 14|\overline{x}_1-\overline{x}_j|,  \quad \text {if} ~|y-\overline{x}_1|\ge \frac 14|\overline{x}_1-\overline{x}_j|,
\end{align*}
\begin{align*}
|y-\underline{x}_{i}| \ge \frac 14|\overline{x}_1-\underline{x}_1|\ge C\Big(\frac r k\Big) .
\end{align*}
Then
\begin{align*}
 \Big(\sum_{j=2}^k  U_{\overline{x}_j, \Lambda}+\sum_{j=1}^k U_{\underline{x}_j, \Lambda}\Big)
 &\le\frac{C}{\big(1+|y-\overline{x}_{1}|\big)^{N-2-\alpha}} \Big[\sum_{j=2}^k  \frac{1}{\big(1+|y-\overline{x}_{j}|\big)^{\alpha}}+\frac{1}{\big(1+|y-\underline{x}_{1}|\big)^{\alpha}}\Big]
\\[2mm]
&\le \frac{C}{\big(1+|y-\overline{x}_{1}|\big)^{N-2-\alpha}}  \Big[\sum_{j=2}^k \frac{1}{|\overline{x}_1-\overline{x}_j|^{\alpha}}+\frac{1}{|\overline{x}_1-\underline{x}_1|^{\alpha}}\Big]
\\[2mm]
&\le \frac{C}{\big(1+|y-\overline{x}_{1}|\big)^{N-2-\alpha}} \Big(\frac{k}{{\bf r}}\Big)^\alpha .
\end{align*}
 \end{proof}

\begin{lem} \label{b.01}
Under the condition  $(r,h,\Lambda) \in {{\mathscr  S}_k} $, for  $ y\in \Omega_1^{+}  $  we have
\begin{align*}
   \Big(  \sum_{j=2 }^k  \overline{\mathbb{Z}}_{2j}  +   \sum_{j=1 }^k   \underline{\mathbb{Z}}_{2j}   \Big)
  \le  C\Big(\frac{k}{{\bf r}}\Big)^\alpha  \frac{\bf r}{\big(1+|y-\overline{x}_{1}|\big)^{N-1-\alpha}},
\end{align*}
 with  $ \alpha=(1, N-1)  $.
 \end{lem}
\begin{proof}
The proof of Lemma \ref{b.01} is similar to Lemma \ref{b.0}. We omit the details for concise.
\end{proof}

For each fixed  $i $ and  $j $,  $i\neq j $, we consider the following
function
\begin{equation*}
g_{ij}(y)=\frac{1}{(1+|y-x_{j}|)^{\gamma_{1}}}\frac{1}{(1+|y-x_{i}|)^{\gamma_{2}}},
\end{equation*}
where  $\gamma_{1}\geq 1 $ and  $\gamma_{2}\geq 1 $ are two constants.
 \begin{lem}\label{lemb1}(Lemma B.1, \cite{WY-10-JFA})
 For any constants  $0<\upsilon\leq \min\{\gamma_{1},\gamma_{2}\} $, there is a constant  $C>0 $, such that
  $$
 g_{ij}(y)\leq \frac{C}{|x_{i}-x_{j}|^{\upsilon}}\Big(\frac{1}{(1+|y-x_{i}|)^{\gamma_{1}+\gamma_{2}-\upsilon}}+\frac{1}{(1+|y-x_{j}|)^{\gamma_{1}+\gamma_{2}-\upsilon}}\Big).
  $$
 \end{lem}

\begin{lem}\label{lemb2}(Lemma B.2, \cite{WY-10-JFA})
For any constant  $0<\beta <N-2 $, there is a constant  $C>0 $, such
that
 $$
\int_{\R^N} \frac{1}{|y-z|^{N-2}}\frac{1}{(1+|z|)^{2+\beta}}{\mathrm d}z \leq
\frac{C}{(1+|y|)^{\beta}}.
 $$
\end{lem}

\begin{lemma}\label{laa3}

Suppose that  $N\ge 5 $ and  $\tau\in(0, 2), y=(y_1, \cdots, y_N)  $. Then there is a small
 $\sigma>0 $, such that
when  $ y_3 \geq 0  $,
\[
\begin{split}
&\int_{\R^N}\frac1{|y-z|^{N-2}} W_{r,h,\Lambda}^{\frac4{N-2}}(z)\sum_{j=1}^k\frac1{(1+|z-\overline{x}_j|)^{\frac{N-2}{2}+\tau}}\, {\mathrm d}z \\[2mm]
 & \le C\sum_{j=1}^k\frac1{(1+|y-\overline{x}_j|)^{\frac{N-2}{2}+\tau+\sigma}},
\end{split}
\]
 and when  $  y_3 \le 0 $,
\[
\begin{split}
&\int_{\R^N}\frac1{|y-z|^{N-2}} W_{r,h,\Lambda}^{\frac4{N-2}}(z)\sum_{j=1}^k\frac1{(1+|z-\underline{x}_j|)^{\frac{N-2}{2}+\tau}}\, {\mathrm d}z \\[2mm]
 &\le C\sum_{j=1}^k\frac1{(1+|y-\underline{x}_j|)^{\frac{N-2}{2}+\tau+\sigma}}.
\end{split}
\]

\end{lemma}

\begin{proof}
The proof of Lemma \ref{laa3} is similar to Lemma B.3  in \cite{WY-10-JFA}.  Here we omit it.
\end{proof}
%

\begin{lemma}   \label{B6}
Suppose that  $N\ge 5 $ and $m$ satisfies \eqref{assumptionform}.  We have
\begin{align}  \label{eee}
 {\bf r}  \max\Big\{ \frac1{k^{(\frac m{N-2-m})(
N+2-2\frac{N-2-m}{N-2}- 2\epsilon_1 )}},  \frac1{k^{(\frac {N-2}{N-2-m}) \min\{2m, m+3\}}} \Big\}\le     \frac{C}{k^{\big(\frac{m(N-2)}{N-2-m}+\frac{(N-3)}{N-1}+\sigma\big)}},
\end{align}
provided with   $\sigma, \epsilon_1$ small enough.
\end{lemma}
\begin{proof}    It's easy to show that
\begin{align*}
 \frac{\bf r}{k^{(\frac {N-2}{N-2-m}) \min\{2m, m+3\}}} \le     \frac{C}{k^{\big(\frac{m(N-2)}{N-2-m}+\frac{(N-3)}{N-1}+\sigma\big)}},
\end{align*}
for $m\ge 2$.  In order to get \eqref{eee}, we just need to show
\begin{align}   \label{eeee}
   \frac {\bf r} {k^{(\frac m{N-2-m})(
N+2-2\frac{N-2-m}{N-2}- 2\epsilon_1 )}}  =    \frac{k^{\frac{N-2}{N-2-m} }}{k^{(\frac m{N-2-m})(
N+2-2\frac{N-2-m}{N-2}- 2\epsilon_1 )}}\le     \frac{C}{k^{\big(\frac{m(N-2)}{N-2-m}+\frac{(N-3)}{N-1}+\sigma\big)}},
\end{align}      for some $\sigma, \epsilon_1$ small.  The problem to show \eqref{eeee} can be reduced to  show  that $~ 6+   \frac{(N-3)}{N-1}  <  3(\frac{N-2} {N-2-m} ) + 2\frac{N-2-m}{N-2}$,  for $m$ satisfying  \eqref{assumptionform}.  This inequality follows by simple computations. This fact concludes the proof.
\end{proof}

\medskip

{\bf Acknowledgements:}
L. Duan  was supported by the China Scholarship Council and NSFC grant (No.11771167) and  Technology Foundation of Guizhou Province ($[2001]$ZK008). M. Musso was supported by  EPSRC research Grant EP/T008458/1. S. Wei was supported by the NSFC grant (No.12001203) and Guangdong Basic and Applied Basic Research Foundation (No. 2020A1515110622). Some part of the work  was done  during the visit of  L. Duan  to Prof. M. Musso at the University of Bath. L. Duan would like to thank the Department of Mathematical Sciences for its warm hospitality and supports.

 \end{document}